\documentclass[11pt]{article}
\usepackage{xcolor}
\usepackage{mathrsfs}
\usepackage{amsmath,amsthm,amssymb,amscd}
\usepackage{booktabs} 
\usepackage{float}
\usepackage[hyperfootnotes=false, colorlinks, linkcolor={blue}, citecolor={magenta}, filecolor={blue}, urlcolor={blue}, plainpages=false, pdfpagelabels]{hyperref}
\usepackage[overload]{textcase} 
\usepackage{mathtools}
\usepackage[numbers,sort&compress]{natbib} 
\usepackage{etoolbox} 
\apptocmd{\sloppy}{\hbadness 10000\relax}{}{} 
\usepackage[left=2.5cm, right=2.5cm, top=3cm]{geometry}
\allowdisplaybreaks
\usepackage{enumerate}
\usepackage[shortlabels]{enumitem}
\usepackage{bm}
\usepackage{url}
\usepackage[multiple]{footmisc}

\usepackage{colonequals}
\usepackage{longtable} 
\usepackage[ampersand]{easylist}
\usepackage{ltablex,booktabs}
\usepackage{comment}
\usepackage[normalem]{ulem}

\def\re{\operatorname{Re}}
\def\spin{\operatorname{spin}}

\def\mo{\operatorname{mod}}

\def\det{\operatorname{det}}

\newcommand{\A}{{\mathbb A}}
\newcommand{\Q}{{\mathbb Q}}
\newcommand{\Z}{{\mathbb Z}}
\newcommand{\R}{{\mathbb R}}
\newcommand{\C}{{\mathbb C}}

\newcommand{\GL}{{\rm GL}}

\newcommand{\SL}{{\rm SL}}
\newcommand{\SO}{{\rm SO}}
\newcommand{\Sp}{{\rm Sp}}

\newcommand{\GSp}{{\rm GSp}}
\newcommand{\PGSp}{{\rm PGSp}}

\newcommand{\Gal}{{\rm Gal}}

\newcommand{\trace}{{\rm tr}}

\newcommand{\sym}{{\rm sym}}
\newcommand{\Frob}{{\rm Frob}}
\def\re{\operatorname{Re}}

\newcommand{\ds}{\displaystyle}

\newcommand{\set}[1]{\left\{#1\right\}}
\newcommand{\abs}[1]{\left|#1\right|}

\newcommand{\forget}[1]{}
\def\qdots{\mathinner{\mkern1mu\raise0pt\vbox{\kern7pt\hbox{.}}\mkern2mu
\raise3.4pt\hbox{.}\mkern2mu\raise7pt\hbox{.}\mkern1mu}}

\newtheorem{lemma}{Lemma}[section]
\newtheorem{theorem}[lemma]{Theorem}

\newtheorem{proposition}[lemma]{Proposition}
\newtheorem{definition}[lemma]{Definition}
\newtheorem{remark}{Remark}

\newtheorem{conjecture}[lemma]{Conjecture}
\newtheorem{hypothesis}[lemma]{Hypothesis}

\newcommand\blfootnote[1]{%
	\begingroup
	\renewcommand\thefootnote{}\footnote{#1}%
	\addtocounter{footnote}{-1}%
	\endgroup
}
\usepackage[toc,page, title, titletoc]{appendix}
\makeatother
\makeatletter
\newcommand\appendix@section[1]{%
	\refstepcounter{section}%
	\orig@section*{Appendix \@Alph\c@section: #1}%
	\addcontentsline{toc}{section}{Appendix \@Alph\c@section: #1}%
}
\g@addto@macro\appendix{\let\section\appendix@section}
\let\orig@section\section
\makeatother
\makeatletter
\def\thickhline{%
  \noalign{\ifnum0=`}\fi\hrule \@height \thickarrayrulewidth \futurelet
   \reserved@a\@xthickhline}
\def\@xthickhline{\ifx\reserved@a\thickhline
               \vskip\doublerulesep
               \vskip-\thickarrayrulewidth
             \fi
      \ifnum0=`{\fi}}
\makeatother

\title{Some remarks on strong multiplicity one for paramodular forms}
\author{Xiyuan Wang, Zhining Wei, Pan Yan and Shaoyun Yi}
\date{}

\begin{document}

\maketitle
\blfootnote{
\\ \hspace*{0.22in} Key words and phrases. Siegel modular forms, paramodular forms, $L$-functions, strong multiplicity one}
\begin{abstract}
We establish several refined strong multiplicity one results for paramodular cusp forms by using automorphic and Galois-theoretic methods. We also give an application to distinguishing eigenforms by the twisted central values of the spinor $L$-functions, which is based on a result in Radziwi{\l}{\l} and Yang \cite{RadziwillYang2023}.
\end{abstract}

\tableofcontents

\section{Introduction}\label{section intro}
Let $\mathbb{A}$ be the ring of adeles of a number field, and let $\pi\simeq\otimes_v^\prime\pi_v$ and $\pi'\simeq\otimes_v^\prime\pi_v'$ be two cuspidal automorphic representations of $\GL_m(\A)$ and $\GL_{m'}(\A)$, respectively. The strong multiplicity one theorem asserts that if $\pi_v\simeq\pi_v'$ for all but finitely many places $v,$ then $\pi\simeq\pi'$. This is first proved by Piatetski-Shapiro \cite{Piatetski-Shapiro1979} using the uniqueness of the Kirillov model and then by Jacquet and Shalika \cite{JacquetShalika198101} using Rankin-Selberg $L$-functions. Later some quantitative strong multiplicity one results are obtained; see for example \cite{Moreno1985, GoldfeldHoffstein1993, Brumley2006, LiuWang2009}.
On the other hand, by Hecke theory the isomorphism of local representations is equivalent to the equality of Hecke eigenvalues. Then a natural question is, \textit{whether we can distinguish two automorphic forms by a set of eigenvalues}? This can be viewed as a refinement of the strong multiplicity one theorem. For the case of classical elliptic modular forms, the problem of distinguishing two Hecke eigenforms by Hecke eigenvalues has been well studied; see for example \cite{Sturm1987, RamMurty1997, Gh2011, VilardiXue2018}.

In this paper, we work exclusively over the number field $\Q$, and consider the strong multiplicity one problem for the symplectic similitude group $\GSp_4$. This kind of strong multiplicity one result has been established for generic cuspidal automorphic representations of $\GSp_4$ (see \cite{Soudry1987}) as well as certain paramodular forms (see \cite{RosnerWeissauer2017}). However, it turns out that for holomorphic Siegel modular forms of degree $2$, the non-generic nature of the underlying archimedean representation prevents the direct application of standard techniques in automorphic forms. Using Arthur's work on a classification of the discrete automorphic spectrum in terms of automorphic representations of general linear groups, Schmidt \cite{Schmidt2018, Schmidt2020} first proves a strong multiplicity one result for holomorphic paramodular cusp forms of general level. This result has been improved by Kumar, Meher and Shankhadhar \cite{KumarMeherShankhadhar2021} in the full level case. More precisely, they essentially show that any set of eigenvalues (normalized or non-normalized) at primes $p$ of positive upper density are sufficient to determine the Siegel cuspidal eigenform of full level; see \cite[Theorems~1.5, 1.6]{KumarMeherShankhadhar2021}. Recently, the second and fourth authors \cite{wei2022distinguishing} of this article further investigate the question of distinguishing Siegel cusp forms of degree $2$ from various aspects with several improved results, particularly including the approaches from $L$-functions. 

In this work, we continue to apply the theory of $L$-functions to establish more general results towards holomorphic paramodular forms. More precisely, we consider holomorphic paramodular forms of general level by means of two different kinds of $L$-functions, i.e., the spinor $L$-functions of degree $4$ (see \eqref{finite part of spinor L function}) and standard $L$-functions of degree $5$ (see \eqref{finite part of standard L function}). To state our main results properly, we first briefly recall some necessary definitions. 

For a positive integer $N$, the paramodular group of level $N$ is defined as
\begin{equation}\label{paramodular of level N}
    K(N)\coloneqq\Sp_4(\Q)\cap
    \begin{bsmallmatrix}
    \Z&N\Z&\Z&\Z\\
    \Z&\Z&\Z&N^{-1}\Z\\
    \Z&N\Z&\Z&\Z\\
    N\Z&N\Z&N\Z&\Z
    \end{bsmallmatrix}.
\end{equation}
Siegel modular forms of degree $2$ with respect to the paramodular group $K(N)$ are called paramodular forms of level $N$, and they have received much attention in recent years because of their appearance in the paramodular conjecture; see for example \cite{BrumerKramer2014, BrumerKramer2019, PoorYuen2015, BoxerCalegariGeePilloni2021}. On the other hand, there is also a satisfactory local theory of paramodular fixed vectors, developed in \cite{RobertsSchmidt2007}, with properties similar to the familiar local newform theory for $\GL_2$ as in \cite{Casselman1973}. In fact, this kind of local theory results in a global theory of paramodular oldforms and newforms; see \cite{RobertsSchmidt2006}. Let $\mathcal{S}_{k, j}(K(N))$ be the spaces of paramodular cusp forms of weight $(k, j)\in\Z_{>0}\times\Z_{\geq 0}$, and let $\mathcal{S}_{k, j}^{\mathrm{new}}(K(N))$ be the subspace of newforms. It is well known that we can associate with a Hecke eigenform $F\in  \mathcal{S}_{k, j}^{\mathrm{new}}(K(N))$ a cuspidal automorphic representation $\pi_F\simeq\otimes_{v}^\prime \pi_{F,v}$ of $\GSp_4(\A)$ with trivial central character; see for example \cite[\S~2.1]{RoySchmidtYi2021}. 

Our first main result in the following shows that the eigenvalues of Hecke operators $T(p)$ for almost all primes $p$ are sufficient to distinguish paramodular newforms.

\begin{theorem}[\text{Theorem~\ref{main thm 1}}]\label{First Main Theorem}
Let $(k_i, j_i)\in\Z_{>0}\times\Z_{\geq 0}$ and $N_i\in\Z_{>0}$, and let $F_i\in \mathcal{S}_{k_i, j_i}^{\mathrm{new}}(K(N_i))$ be two Hecke eigenforms for $i=1, 2$. For $\re(s)>3/2$, let 
\begin{equation*}
L(s,\pi_{F_i},\rho_4)=\sum_{n\geq 1}\frac{a_{F_i}(n)}{n^s},\quad i=1, 2,
\end{equation*}
be the associated spinor $L$-functions. If $a_{F_1}(p)=a_{F_2}(p)$ for almost all primes $p$, then $F_1=c\cdot F_2$ for some nonzero constant $c$.
\end{theorem}

\begin{remark}\label{rk 1 after main thm 1}
Theorem~\ref{First Main Theorem} is
a refinement of \cite{Schmidt2018, Schmidt2020}, in which Schmidt shows that paramodular newforms enjoy a strong multiplicity one property in terms of the spinor $L$-functions with respect to the eigenvalues of two Hecke operators $T(p)$ and $T(p^2)$ for almost all primes $p$. On the other hand, Theorem~\ref{First Main Theorem} also generalizes the relevant results of \cite{FarmerPitaleRyanSchmidt2013arxiv}, in which they only consider Siegel Hecke eigenforms of full level.
\end{remark}

Our second main theorem gives an analogous result with respect to the standard $L$-functions under a certain mild condition. To state this result precisely, we further need the following definition.
We say that a Hecke eigenform $F\in \mathcal{S}_{k, j}(K(N))$ is \textit{good} if local components $\pi_{F,2}$ and $\pi_{F,3}$ of the associated cuspidal automorphic representation $\pi_F$ of $\GSp_4(\A)$ are not supercuspidals and there does not exist a non-trivial character $\chi$ such that $\pi_F\simeq\pi_F\otimes\chi$; see Definition~\ref{definition of good}.

\begin{theorem}[\text{Theorem~\ref{main thm 2}}]\label{Second Main Theorem}
Let $(k_i, j_i)\in\Z_{\geq 3}\times\Z_{\geq 0}$ and $N_i\in\Z_{>0}$ for $i=1, 2$. Let $F_i\in \mathcal{S}_{k_i, j_i}^{\mathrm{new}}(K(N_i)), i=1, 2$, be two Hecke eigenforms, which are \textit{good} as defined above.
For $\re(s)>3/2$, let 
\begin{equation*}
L(s,\pi_{F_i},\rho_5)=\sum_{n\geq 1}\frac{b_{F_i}(n)}{n^s},\quad i=1, 2,
\end{equation*}
be the associated standard $L$-functions. If $b_{F_1}(p)=b_{F_2}(p)$ for almost all primes $p$, then there exists a quadratic character $\chi$ such that $\pi_{F_1}\simeq\pi_{F_2}\otimes\chi$. Additionally, if we assume that $N_1$ and $N_2$ are squarefree numbers, then $F_1=c\cdot F_2$ for some nonzero constant $c$.
\end{theorem}

\begin{remark}\label{rk 2 after main thm 2}
To the best of our knowledge, this kind of strong multiplicity one result in terms of the standard $L$-functions is new. We also note that the only reason requiring the conditions on the local representations at primes $p=2, 3$, in the definition of a Hecke eigenform $F$ being \textit{good} comes from the assumptions in \cite[Theorem~A]{Kim2003}. On the other hand, it would be interesting if one can improve the condition $k\geq 3$ to $k>0$. It is a well known fact that there is no holomorphic paramodular form of weight $1$. Thus, some additional work is needed for the $k=2$ case to achieve such improvement. In this work, our method can only deal with the regular case, meaning that the corresponding Galois representations having distinct Hodge-Tate weights $\{0, k-2, k+j-1, 2k+j-3\}$.
\end{remark}

As an application of Theorem~\ref{First Main Theorem}, our third main result on distinguishing Hecke eigenforms is obtained by using the central values of a family of the twisted spinor $L$-functions. More precisely, let $\chi$ be a primitive Dirichlet character, and for simplicity of notation we will also view $\chi$ as a Hecke character of $\Q^\times\backslash\A^\times$. Let $\sigma_1$ be the standard representation of the dual group $\GL_1(\C)=\C^\times$. Recall that for a Hecke eigenform $F\in  \mathcal{S}_{k, j}^{\mathrm{new}}(K(N))$, we let $L(s, \pi_F, \rho_4)$ be the associated spinor $L$-function as in Theorem~\ref{First Main Theorem}, i.e., $L(s, \pi_F, \rho_4)=\sum_{n\geq1}a_{F}(n)n^{-s}$ for $\re(s)>3/2$.
Then we can define the twisted spinor $L$-function by
\begin{equation}\label{finite part of twisted spinor L function}
L(s,\pi_{F}\times\chi,\rho_4\otimes\sigma_1)=\sum_{n\geq1}\frac{a_{F}(n)\chi(n)}{n^{s}},\quad \re(s)>3/2.
\end{equation}
Note that the twisted spinor $L$-function $L(s,\pi_{F}\times\chi,\rho_4\otimes\sigma_1)$ admits an analytic continuation to the whole complex plane and satisfies a certain functional equation. We denote by $\Gamma_2=K(1)$ the full modular group $\Sp_4(\Z)$.

\begin{theorem}[\text{Theorem~\ref{main thm 3}}]\label{Third Main Theorem}
Let $F_i\in\mathcal{S}_{k_i}(\Gamma_2), i=1, 2$, be two Hecke eigenforms. Suppose that for almost all primitive characters $\chi$ of squarefree conductor, we have
\begin{equation}\label{main assumption intro}
L(1/2,\pi_{F_1}\times\chi,\rho_4\otimes\sigma_1)=L(1/2,\pi_{F_2}\times\chi,\rho_4\otimes\sigma_1),
\end{equation}
then $k_1=k_2$ and $F_1=c\cdot F_2$ for some nonzero constant $c$.    
\end{theorem}

\begin{remark}\label{rk 3 after main thm 3}
First, we note that Theorem~\ref{Third Main Theorem} gives a generalization of \cite[Proposition~1.3]{wei2022distinguishing}, in which only Saito-Kurokawa liftings are considered. Even though, the assumption in \cite[Proposition~1.3]{wei2022distinguishing} only requires for almost all primitive quadratic characters. Secondly, we point out that the same result in Theorem~\ref{Third Main Theorem} should also hold for more general paramodular newforms $F_i\in \mathcal{S}_{k_i, j_i}^{\mathrm{new}}(K(N_i)), i=1, 2$, with $(k_i, j_i)\in \Z_{>0}\times\Z_{\geq 0}, N_i\in\Z_{>0}$ by applying similar arguments in Section~\ref{section-application}. Finally, if we further assume that $k_1, k_2\geq3$, then it suffices to assume that \eqref{main assumption intro} holds for almost all primitive characters of $3$-squarefree conductor. Here, $3$-squarefree means that the conductor can be written as the product of at most $3$ distinct primes. In fact, the condition that $k_1, k_2\geq 3$ are required only for general levels case since the weight of scalar-valued Siegel cusp forms of full level is at least $10$; see for example \cite[(3.6)]{RoySchmidtYi2021}. As an immediate consequence, this implies that the assumption in Theorem~\ref{Third Main Theorem} on the conductor can be refined to $3$-squarefree conductor.
\end{remark}

The paper is organized as follows: In Section~\ref{section-siegel-modular-forms}, we review the basics of Siegel modular forms of degree $2$, as well as their spinor and standard $L$-functions. In Section~\ref{section-spinor}, we first prove several results about the strong multiplicity one theorem of the general linear groups $\GL_m$ and then give an application to distinguishing paramodular cusp forms in terms of the spinor $L$-functions. In Section~\ref{section-standard}, we first show the cuspidality of automorphic representations associated to the standard $L$-functions and then we give a result of distinguishing paramodular cusp forms in terms of such $L$-functions from Galois representation theoretic approach. Finally, Section~\ref{section-application} is devoted to an application to distinguishing   Hecke eigenforms of degree $2$ by using the central values of a family of the twisted spinor $L$-functions.

\subsection*{Notation} 
All algebraic groups are defined over $\Q$ in this paper. We use the notation $A\ll_{x, y, z} B$ to indicate that there exists a positive constant $C$, depending on $x, y, z$, so that $|A|\leq C|B|$. The symbol $\epsilon$ will denote a small positive number. We write $A(x)=O_y(B(x))$ if there exists a positive real number $M$ (depending on $y$) and a real number $x_0$ such that $|A(x)|\leq M|B(x)|$ for all $x\geq x_0$. We write $a\sim A$ to mean that $a\in [A, 2A)$.

\section*{Acknowledgments}
The authors are deeply grateful to the referee for a very careful review and a number of insightful comments and suggestions, which improve this paper. We thank Wenzhi Luo, Ralf Schmidt, Ariel Weiss and Liyang Yang for their helpful discussions and comments. Pan Yan is supported by an AMS-Simons Travel Grant. Shaoyun Yi is supported by the National Natural Science Foundation of China (No. 12301016) and the Fundamental Research Funds for the Central Universities (No. 20720230025).

\section{Siegel modular forms}\label{section-siegel-modular-forms}
The algebraic group $\GSp_4$ is defined as
\begin{equation}
\GSp_4 \coloneqq \{g\in\GL_4\colon \:^tgJg=\mu(g)J,\:\mu(g)\in \GL_1\}, \quad J=\begin{bsmallmatrix} &&&1\\&&1&\\&-1&&\\-1&&&\end{bsmallmatrix}.
\end{equation}
The character $\mu\colon\GSp_4\to\GL_1$ is called the multiplier homomorphism. The kernel of $\mu$ is the symplectic group $\Sp_4$. Let $\mathrm{Z}$ be the center of $\GSp_4$ and $\PGSp_4=\GSp_4/\mathrm{Z}$. When speaking about Siegel modular forms of degree $2$, it is more convenient to realize symplectic groups using the symplectic form $J=\begin{bsmallmatrix} 0&1_2\\-1_2&0\end{bsmallmatrix}$. Let $\mathbb{H}_2$ be the Siegel upper half plane of degree $2$, consisting of all symmetric complex $2\times 2$ matrices $Z$ whose imaginary part is positive definite.
The group $\GSp_4(\mathbb{R})^+\coloneqq\{g\in\GSp_4(\R)\colon \mu(g)>0\}$ acts on $\mathbb{H}_2$ by the linear fractional transformation
\begin{equation}
    g\langle Z\rangle\coloneqq (AZ+B)(CZ+D)^{-1}\quad \text{for } g=\begin{bsmallmatrix} A&B\\C&D\end{bsmallmatrix}\in \GSp_4(\mathbb{R})^+ \text{ and } Z\in\mathbb{H}_2.
\end{equation}
Let $k\in\Z$ and $j\in\Z_{\geq 0}$. Let $U_j\simeq\mathrm{sym}^j(\C^2)$ be the space of all complex homogeneous polynomials of degree $j$ in the two variables $S$ and $T$. For $h\in\GL_2(\C)$ and $P(S, T)\in U_j$ we define
\begin{equation}
\eta_{k, j}(h)P(S, T)=\det(h)^kP\big((S, T)h\big).
\end{equation}
Then $(\eta_{k, j}, U_j)$ gives a concrete realization of the irreducible representation $\det^k\mathrm{sym}^j$ of $\GL_2(\C)$.

For a positive integer $N$, let $K(N)$ be the paramodular group of level $N$ defined as in \eqref{paramodular of level N}. Let $\mathcal{M}_{k, j}(K(N))$ be the space of paramodular form of level $N$ and weight $(k, j)$ (i.e., $(k+j, k)$ is the associated minimal $K$-type; see \cite[p.~2402]{Schmidt2017}) with respect to $K(N)$, and let $\mathcal{S}_{k, j}(K(N))$ be the subspace of cusp forms, which can be defined as usual. To give more details, a function $F\in \mathcal{M}_{k, j}(K(N))$ is a holomorphic $U_j$-valued function on $\mathbb{H}_2$ satisfying $\big(F|_{k, j}\gamma\big)(Z)=F(Z)$ for all $\gamma\in K(N)$. Here, the slash operator $|_{k, j}$ on $F$ is defined as
\begin{equation}
    \big(F|_{k, j} g\big)(Z)\coloneqq\mu(g)^{k+\frac{j}{2}}\eta_{k, j}(CZ+D)^{-1}F(g\langle Z\rangle)\quad \text{for } g=\begin{bsmallmatrix} A&B\\C&D\end{bsmallmatrix}\in \GSp_4(\mathbb{R})^+ \text{ and } Z\in\mathbb{H}_2.
\end{equation}
The normalization factor $\mu(g)^{k+\frac{j}{2}}$ ensures that the center of $\GSp_4(\mathbb{R})^+$ acts trivially. We abbreviate $\mathcal{M}_{k, 0}(K(N))$ as $\mathcal{M}_{k}(K(N))$ and $\mathcal{S}_{k, 0}(K(N))$ as $\mathcal{S}_{k}(K(N))$, respectively. These are the spaces of $\C$-valued paramodular forms, respectively, cusp forms. Moreover, we denote by $\mathcal{M}_{k}^{(1)}(\Gamma)$ the space of modular forms of degree $1$ and weight $k$ with respect to the congruence subgroup $\Gamma$ of $\SL_2(\Z)$, and by $\mathcal{S}_{k}^{(1)}(\Gamma)$ its subspace of cusp forms.

Let $N, k\in\Z_{>0}$ and $j\in\Z_{\geq 0}$. Let $F\in \mathcal{S}_{k, j}^{\mathrm{new}}(K(N))$ be a Hecke eigenform, i.e., it is an eigenvector for all the Hecke operator $T(n), (n, N)=1$. Denote by $\lambda_F(n)$ the eigenvalue of $F$ under the Hecke operator $T(n)$ when $(n,N)=1$. For any prime $p\nmid N$, we let $\alpha_{p,0},\alpha_{p,1},\alpha_{p,2}$ be the classical Satake parameters of $F$ at $p$. It is well known that
\begin{equation}
    \alpha_{p,0}^2\alpha_{p,1}\alpha_{p,2}=p^{2k-3}.
\end{equation}
Moreover, using these Satake parameters we can define the local spinor $L$-factor $L_p(s,F,\spin)$ of $F$ as follows.
\begin{equation}\label{localspin}
  L_p(s,F,\spin)^{-1}=(1-\alpha_{p,0}p^{-s})(1-\alpha_{p,0}\alpha_{p,1}p^{-s})(1-\alpha_{p,0}\alpha_{p,2}p^{-s})(1-\alpha_{p,0}\alpha_{p,1}\alpha_{p,2}p^{-s}).
\end{equation}
In fact, by \cite{Andrianov1974} one can rewrite $L_p(s,F,\spin)$ in terms of Hecke eigenvalues at $p$ and $p^2$, i.e., 
\begin{equation}\label{eisp}
L_p(s,F,\spin)^{-1}=1-\lambda_F(p)p^{-s}+(\lambda_F(p)^2-\lambda_F(p^2)-p^{2k-4})p^{-2s}-\lambda_F(p)p^{2k-3-3s}+p^{4k-6-4s}.
\end{equation}
In particular, by comparing \eqref{localspin} with \eqref{eisp} we have
\begin{align}
\lambda_F(p)&=p^{k-3/2}(\alpha_p+\alpha_p^{-1}+\beta_p+\beta_p^{-1})\label{lambda p},\\
\lambda_F(p^2)&=p^{2k-3}\left((\alpha_p+\alpha_p^{-1})^2+(\alpha_p+\alpha_p^{-1})(\beta_p+\beta_p^{-1})+(\beta_p+\beta_p^{-1})^2-2-1/p\right),\label{lambda p2}    
\end{align}    
where $\alpha_p=p^{3/2-k}\alpha_{p,0}$ and $\beta_p=\alpha_p\alpha_{p,1}$.
Let $S=\{p<\infty\colon p\mid N\}$ be a finite set of primes. Then we can define the partial spinor $L$-function
\begin{equation}
    L^S(s,F,\spin)=\prod_{p<\infty,\:p\notin S}L_p(s,F,\spin).
\end{equation}
Note that for $p\mid N$ we can also define the local $L$-factor $L_p(s, F,\spin)$ (see \cite{Schmidt2018}) and then we have the spinor $L$-function defined as an Euler product over all finite places by
\begin{equation}
L(s,F,\spin)=\left(\prod_{p\in S}L_p(s,F,\spin)\right)L^S(s,F,\spin).
\end{equation}
Moreover, we can associate with $F$ a cuspidal automorphic representation $\pi_F$ of $\GSp_4(\A)$ with trivial central character. Let $\rho_4$ be the $4$-dimensional irreducible representation of $\Sp_4(\C)$. In fact, $\rho_4$ is the natural representation of $\Sp_4(\C)$ on $\C^4$, which is also called the spin representation. For $\pi_F $ we can define the associated spinor $L$-function $L(s,\pi_F,\rho_4)$ of degree $4$ by
\begin{equation}\label{finite part of spinor L function}
L(s,\pi_F,\rho_4)=L(s+k-3/2,F,\spin)=\sum_{n=1}^{\infty}\frac{a_F(n)}{n^s},\quad\text{for } \re(s)>3/2.
\end{equation}
Note that we will use the notation $\Lambda(s,\pi_F,\rho_4)$ for the completed spinor $L$-function, which also contains the archimedean $L$-factors $L_\infty(s,\pi_F,\rho_4)$. 

On the other hand, let $\rho_5$ be the $5$-dimensional irreducible representation of $\Sp_4(\C)$. An explicit formula for the representation $\rho_5$ as a map $\Sp_4(\C)\to \SO_5(\C)$ is given in \cite[Appendix~A.7]{RobertsSchmidt2007}. Using the notation above, we can similarly define the standard $L$-function of degree $5$ by
\begin{equation}\label{finite part of standard L function}
    L(s,\pi_F, \rho_5)=\left(\prod_{p\in S}L_p(s, F, \mathrm{std})\right)L^S(s, F, \mathrm{std})=\sum_{n=1}^{\infty}\frac{b_F(n)}{n^s},\quad\text{for } \re(s)>3/2.
\end{equation}
Here, the local $L$-factor $L_p(s,F,\mathrm{std})$ for $p\mid N$ is as in \cite{Schmidt2018} and the partial standard $L$-function is defined by
\begin{equation}
L^S(s,F,\mathrm{std})=\prod_{p<\infty,\:p\notin S}L_p(s,F,\mathrm{std})
\end{equation}
with 
\begin{equation}\label{local standard L function}  
   L_p(s,F,\mathrm{std})^{-1}=(1-p^{-s})(1-\alpha_{p,1}p^{-s})(1-\alpha_{p,2}p^{-s})(1-\alpha_{p,1}^{-1}p^{-s})(1-\alpha_{p,2}^{-1}p^{-s}).
 \end{equation}
Again, we will use the notation $\Lambda(s,\pi_F,\rho_5)$ for its completed standard $L$-function. In fact, there is another way to interpret this representation $\rho_5$. More precisely, we have
\begin{equation}\label{ext rho 4 equals 1 plus rho 5}
    \Lambda^2\rho_4=\mathbf{1}\oplus\rho_5,
\end{equation}
where $\mathbf{1}$ is the trivial representation. Another useful relation for our purpose is
\begin{equation}\label{ext 5 equal sym 4}
\mathrm{sym}^2\rho_4=\Lambda^2\rho_5.
\end{equation}

We end this section with a brief introduction to Arthur packets, which give a classification for the discrete spectrum of $\GSp_4(\A)$. First, we point out that only representations with trivial central character, i.e., cuspidal automorphic representations of $\PGSp_4(\A)$, are considered in this work. Observing the exceptional isomorphism of the algebraic group $\PGSp_4$ with the split orthogonal group $\SO_5$, we can identify representations of $\SO_5(\A)$ with representations of $\PGSp_4(\A)$. Then it follows from \cite{Arthur2004} (see also \cite[pp.~3088-3089]{Schmidt2018}) that there are six different types of automorphic representations of $\PGSp_4(\A)$ in the discrete spectrum. Moreover, since we are interested in the paramodular forms, by \cite[Corollary~5.4]{Schmidt2020} only the type {\bf (P)} (consisting of Saito-Kurokawa liftings) and the type {\bf (G)} (the ``general" class, which are characterized by admitting a functorial transfer to a cuspidal automorphic representation of $\GL_4(\A)$) could happen in these forms, and only the type {\bf (G)} shows up if we consider the vector-valued paramodular forms. In particular, we sometimes write $\mathcal{S}^{\text{{\bf (G)}}}_{k, j}(K(N))$ to mean the subspace of $\mathcal{S}_{k, j}(K(N))$ spanned by all eigenforms of type {\bf (G)}, and similarly for the type {\bf (P)}. In summary, we have
\begin{equation}\label{para decomp to G and P types}
\mathcal{S}_{k, j}(K(N))=\mathcal{S}^{\text{{\bf (G)}}}_{k, j}(K(N))\oplus\mathcal{S}^{\text{{\bf (P)}}}_{k, j}(K(N)),\quad k, N\in\Z_{>0}, j\in\Z_{\geq 0},
\end{equation}
and 
\begin{equation}\label{para decomp to G and P types only G}
\mathcal{S}^{\text{{\bf (P)}}}_{k, j}(K(N))=\mathbf{0},\quad j>0.
\end{equation}
Moreover, by \eqref{para decomp to G and P types} and \cite[Table~2]{Schmidt2018} we can see that the spinor and standard $L$-functions (see \eqref{finite part of spinor L function} and \eqref{finite part of standard L function}) associated to paramodular forms are absolutely convergent for $\re(s)>3/2$. See \cite{Arthur2004, Schmidt2018, Schmidt2020} for more details about Arthur packets.  

\section{Strong multiplicity one in terms of the spinor \texorpdfstring{$L$}{L}-functions}
\label{section-spinor}
In this section, our main goal is to distinguish paramodular newforms by using the spinor $L$-functions. More precisely, we will first obtain some results on the strong multiplicity one theorem of general linear group $\GL_m$ (Theorem~\ref{SMO for GLn by L series}) and then apply the special case $m=4$ to distinguishing paramodular newforms as desired (Theorem~\ref{main thm 1}).
\subsection{Certain results on strong multiplicity one theorem of \texorpdfstring{$\GL_m$}{GLm}}\label{subsect of smo for GLm}
Let $\pi$ be a unitary cuspidal automorphic representation of $\GL_m(\A)$. For $\re(s)>1$, we associate with $\pi$ the $L$-function defined by the $L$-series
\begin{equation}
L(s,\pi)=\sum_{n=1}^{\infty}\frac{\lambda_{\pi}(n)}{n^s}.
\end{equation}

In this subsection, we mainly investigate the following conjecture on distinguishing cuspidal automorphic representations of the general linear groups from several aspects.

\begin{conjecture}\label{conjecture of SMO for unitary cuspidal automorphic repns}
Let $\pi$ and $\pi'$ be self-dual unitary cuspidal automorphic representations of $\GL_m(\A)$ and $\GL_{m'}(\A)$, respectively. Assume that $\lambda_{\pi}(p)=\lambda_{\pi'}(p)$ for almost all primes $p$. Then $\pi\simeq\pi'$.
\end{conjecture}

Here are several observations that support the Conjecture~\ref{conjecture of SMO for unitary cuspidal automorphic repns}. 
First, since we assume that such two representations are self-dual, Conjecture~\ref{conjecture of SMO for unitary cuspidal automorphic repns} is actually an immediate consequence of the Selberg orthogonality conjecture for a pair of two self-dual cuspidal automorphic representations, which is stated as follows.
\begin{conjecture}[Selberg's orthogonality conjecture]\label{Selberg orthogonality conjecture}
Let $\pi$ and $\pi'$ be self-dual unitary cuspidal automorphic representations of $\GL_m(\A)$ and $\GL_{m'}(\A)$, respectively. Then for any $x>0$,
\begin{equation}\label{Selberg orthogonality arbitrary}
    \sum_{p\leq x,p\notin S}\frac{\lambda_{\pi}(p)\lambda_{\pi'}(p)}{p}=
    \begin{cases}
       \log\log x+O_{\pi,\pi',S}(1) & \mbox{if $\pi\simeq\pi'$}, \\
         O_{\pi,\pi',S}(1)& \mbox{otherwise}.
    \end{cases}
\end{equation}
Here, $S$ is a finite set of primes. 
\end{conjecture}
\begin{lemma}\label{lemma SOC implies main Conjecture}
Selberg's orthogonality conjecture implies Conjecture~\ref{conjecture of SMO for unitary cuspidal automorphic repns}.
\end{lemma}
\begin{proof}
Since $\lambda_{\pi}(p)=\lambda_{\pi'}(p)$ for almost all primes $p$, we have
\[\sum_{p\leq x,p\notin S}\frac{\lambda_{\pi}(p)\lambda_{\pi'}(p)}{p}=\sum_{p\leq x,p\notin S}\frac{\lambda_{\pi}(p)\lambda_{\pi}(p)}{p}+O(1)\overset{\eqref{Selberg orthogonality arbitrary}}=\log\log x+O(1)\]
for sufficiently large $x$. This implies $\pi\simeq\pi'$ by applying \eqref{Selberg orthogonality arbitrary} again.
\end{proof}
On the other hand, as in the following proposition, we will show that Conjecture~\ref{conjecture of SMO for unitary cuspidal automorphic repns} also follows from the Generalized Ramanujan Conjecture.
\begin{proposition}\label{SMO for GLn by L series under GRC}
If one of the following conditions holds, then  Conjecture~\ref{conjecture of SMO for unitary cuspidal automorphic repns} is true:
\begin{enumerate}[(i)]
    \item
    The Generalized Ramanujan Conjecture holds for at least one of $\pi$ and $\pi'$.
    \item A partial Ramanujan bound for some $\theta<1/4$ holds for both $\pi$ and $\pi'$.
\end{enumerate}
\end{proposition}
\begin{proof}
Let $S$ be a finite set of primes such that it satisfies the following conditions:
\begin{enumerate}[(a)]
    \item It contains all ramified primes $p$ of $\pi$ and $\pi'$;
    \item For all $p\notin S$, we have that $\lambda_{\pi}(p)=\lambda_{\pi'}(p)$.
\end{enumerate}
Without loss of generality, we assume that $m\geq m'$. For $p\notin S$, let $\{\alpha_{p,i}\}_{i=1}^{m}$ (resp. $\{\beta_{p,j}\}_{j=1}^{m'}$) be the Satake parameters of $\pi$ (resp. $\pi'$) at $p$.\footnote{We write $|\alpha_{p, i}|\leq p^\theta$, for some $\theta<1/2$, to indicate progress toward the
Ramanujan bound ($\theta=0$), referring to this as a ``partial Ramanujan bound".} Then we define the following partial Rankin-Selberg $L$-functions
\begin{equation}\label{Rankin-Selberg pi_1 pi_1}
    L^S(s,\pi\otimes\pi)=\prod_{p\notin S}\prod_{i,j=1}^m\left(1-\frac{\alpha_{p,i}\alpha_{p,j}}{p^s}\right)^{-1},
\end{equation}
and
\begin{equation}\label{Rankin-Selberg pi_1 pi_2}
    L^S(s,\pi\otimes\pi')=\prod_{p\notin S}\prod_{i=1}^m\prod_{j=1}^{m'}\left(1-\frac{\alpha_{p,i}\beta_{p,j}}{p^s}\right)^{-1}.
\end{equation}
It is well known that they are both absolutely convergent for $\re(s)>1$. 
In fact, at any place $p$, following \cite{JacquetShalika198101, JacquetShalika198102}, we can define the local zeta integral
\begin{equation*}
\begin{cases}
\ds \Psi_p(s,W_p,W_p^\prime, \Phi_p) = \int_{N_{m}(\Q_p)\backslash \GL_{m}(\Q_p)}W_p(g)W_p^\prime(g)\Phi_p(\epsilon g)|\det g|_p^s\,dg \ \  &\text{ if } m=m^\prime, \\
\ds \Psi_p(s,W_p,W_p^\prime) = \int_{N_{m'}(\Q_p)\backslash \GL_{m'}(\Q_p)}W_p\left(\begin{bsmallmatrix}
g&0\\
0&I_{m-m'}
\end{bsmallmatrix}\right)W'_p(g)|\det g|_p^{s-\frac{m-m^\prime}{2}}\,dg \ \  &\text{ if } m>m^\prime.
\end{cases}
\end{equation*}
Here, $W_p\in \mathcal{W}(\pi_p,\psi_p)$ and $W_p^\prime\in \mathcal{W}(\pi^\prime_p,\psi_p^{-1})$ are Whittaker functions of $\pi_p$ and $\pi^\prime_p$, respectively, where $\psi_p$ is a non-trivial additive character of $\mathbb{Q}_p$, and $N_m$ is the upper triangular unipotent subgroup of $\GL_m$. In the case $m=m^\prime$,
$\Phi_p\in \mathcal{S}(\mathbb{Q}_p^{m})$ is a Schwartz function, and $\epsilon$ is the row vector $(0, \cdots, 0, 1)$.  
By \cite[Propositions~1.5, 3.17]{JacquetShalika198101} and \cite[Propositions~1.5, 2.6]{JacquetShalika198102}, the local integrals $\Psi_p(s,W_p,W_p^\prime, \Phi_p)$ and $\Psi_p(s,W_p,W_p^\prime)$ converge absolutely for $\mathrm{Re}(s)\ge 1$, and hence they have no poles at $s=1$ at any local place. Moreover, at a finite unramified place $p\not\in S$, we have
\begin{equation*}
\begin{cases}
\ds \Psi_p(s,W_p^0,{W_p^{\prime}}^0, \Phi_p^0) = \prod_{i=1}^m\prod_{j=1}^{m'}\left(1-\frac{\alpha_{p,i}\beta_{p,j}}{p^s}\right)^{-1} \ \  &\text{ if } m=m^\prime, \\
\ds \Psi_p(s,W_p^0,{W_p^{\prime}}^0)  = \prod_{i=1}^m\prod_{j=1}^{m'}\left(1-\frac{\alpha_{p,i}\beta_{p,j}}{p^s}\right)^{-1} \ \  &\text{ if } m>m^\prime,
\end{cases}
\end{equation*}
where $W_p^0,{W_p^{\prime}}^0$ are unramified Whittaker functions normalized such that $W_p^0(I_m)={W_p^{\prime}}^0(I_{m^\prime})=1$, and $\Phi_p^0$ is the characteristic function of $\mathbb{Z}_p^m$. By \cite[Propositions~1.5, 3.17]{JacquetShalika198101} and \cite[Propositions~1.5, 2.6]{JacquetShalika198102}, there exists a choice of data so that the local zeta integral is non-vanishing at $s=1$ at any local place. Since $\pi$ is self-dual, it follows from \cite{JacquetShalika198102} that the completed $L$-function $\Lambda(s,\pi\otimes\pi)$ has a simple pole at $s=1$. Consequently, the partial $L$-function $L^S(s,\pi\otimes\pi)$ also has a simple pole at $s=1$. 

Next, still for $\re(s)>1$, by taking the logarithm of both sides of \eqref{Rankin-Selberg pi_1 pi_1}, we obtain that
\begin{equation}\label{log pi_1 pi_1}
\begin{split}
\log L^S(s,\pi\otimes\pi)&=-\sum_{p\notin S}\sum_{i,j=1}^m\log\left(1-\frac{\alpha_{p,i}\alpha_{p,j}}{p^s}\right)=\sum_{p\notin S}\sum_{i,j=1}^m\sum_{\ell=1}^{\infty}\frac{\alpha_{p,i}^{\ell}\alpha_{p,j}^{\ell}}{\ell p^{\ell s}}\\
&=\sum_{p\notin S}\sum_{i,j=1}^m\frac{\alpha_{p,i}\alpha_{p,j}}{p^{s}}+\sum_{p\notin S}\sum_{\ell=2}^{\infty}\sum_{i,j=1}^m\frac{\alpha_{p,i}^{\ell}\alpha_{p,j}^{\ell}}{\ell p^{\ell s}}\\
&=\sum_{p\notin S}\sum_{i,j=1}^m\frac{\alpha_{p,i}\alpha_{p,j}}{p^{s}}+O(1)=\sum_{p\notin S}\frac{\lambda_{\pi}(p)^2}{p^s}+O(1).
\end{split}
\end{equation}
Here, the $O(1)$ term is due to the condition (i) or (ii). A similar argument shows that
\begin{equation}\label{log pi_1 pi_2}
\log L^S(s,\pi\otimes\pi')=\sum_{p\notin S}\frac{\lambda_{\pi}(p)\lambda_{\pi'}(p)}{p^s}+O(1)=\sum_{p\notin S}\frac{\lambda_{\pi}(p)^2}{p^s}+O(1).
\end{equation}
The second equality in \eqref{log pi_1 pi_2} is due to the assumption that $\lambda_{\pi}(p)=\lambda_{\pi'}(p)$ for all $p\notin S$. By comparing \eqref{log pi_1 pi_1} and \eqref{log pi_1 pi_2}, we can see that $L^S(s,\pi\otimes\pi')$ and hence $\Lambda(s,\pi\otimes\pi')$ go to infinity as $s\to 1^+$, i.e., $\Lambda(s,\pi\otimes\pi')$ has a simple pole at $s=1$. It follows from \cite{JacquetPiatetski-ShapiroShalika1983} that $\pi'\simeq \widetilde{\pi} \simeq \pi$ as desired.
\end{proof}

\begin{remark}
    In the non self-dual cases this is known from the work of Avdispahi\'c and Smajlovi\'c \cite{AvdispahicSmajlovic2010}.
\end{remark}

We first point out that Farmer, Pitale, Ryan and Schmidt \cite{FarmerPitaleRyanSchmidt2013arxiv} have also investigated the strong multiplicity one theorem for $L$-functions with a representation-theoretic approach, in which they assume a partial Ramanujan bound $\theta<1/6$ for both functions, plus an additional condition (i.e., the Dirichlet coefficients at the primes being close to each other); see \cite[Theorem~2.2]{FarmerPitaleRyanSchmidt2013arxiv}.

In fact, under the assumptions in Proposition~\ref{SMO for GLn by L series under GRC}, we can actually prove Selberg's orthogonality conjecture by following \cite[Theorem~5.13]{IwKo2004}. In particular, this indicates that the Generalized Ramanujan Conjecture is stronger than Selberg's orthogonality conjecture. Based on this observation, we will prove Conjecture~\ref{conjecture of SMO for unitary cuspidal automorphic repns} in some special cases by showing Selberg's orthogonality conjecture without assuming the Generalized Ramanujan Conjecture. Before the proof, we shall introduce the following hypothesis proposed by Rudnick and Sarnak \cite{RudnickSarnak1996}. Let $\{\alpha_{p,1},\ldots,\alpha_{p,m}\}$ (resp. $\{\beta_{p,1},\ldots,\beta_{p,m'}\}$) be the Satake parameters for $\pi$ (resp.  $\pi'$) as above. For $\ell\geq 1$, we define
\begin{equation}\label{Eqnapipell}
 a_{\pi}(p^{\ell})=\alpha_{p,1}^{\ell}+\cdots+\alpha_{p,m}^{\ell}\quad\text{and}\quad a_{\pi'}(p^{\ell})=\beta_{p,1}^{\ell}+\cdots+\beta_{p,m'}^{\ell}.   
\end{equation}

\begin{hypothesis}[\text{\cite[Hypothesis~{\bf H}]{RudnickSarnak1996}}]\label{Hypothesis implies Selberg orthogonality}
Let $\pi$ be a self-dual unitary cuspidal automorphic representation of $\GL_m(\A)$. Then for $\ell\geq 2$, we have
 \begin{equation}\label{pi H}
 \sum_{p}\frac{|a_{\pi}(p^{\ell})|^2\log^2p}{p^{\ell}}<\infty.
 \end{equation}
\end{hypothesis}

\begin{theorem}\label{Thm Hypothesis H holds for m leq 4}
Hypothesis~\ref{Hypothesis implies Selberg orthogonality} holds for $m\leq 4$.
\end{theorem}
\begin{proof}
For $m\leq 3$ one can refer to \cite[Proposition~2.4]{RudnickSarnak1996}, and for $m=4$ one can refer to \cite[Appendix~2]{Kim2003}.
\end{proof}

For $m>4$, Hypothesis~\ref{Hypothesis implies Selberg orthogonality} is an easy consequence of the Generalized Ramanujan Conjecture (GRC). On the other hand, it follows from \cite{RudnickSarnak1996} that if $\pi$ satisfies the Hypothesis~\ref{Hypothesis implies Selberg orthogonality}, then Selberg's orthogonality conjecture holds for $\pi$. Combining Lemma~\ref{lemma SOC implies main Conjecture}, we have the following relations
\begin{equation}\label{H implies SOC implies main conjecture}
\text{GRC}\ \Rightarrow \ \text{Hypothesis~\ref{Hypothesis implies Selberg orthogonality}}\ \Rightarrow \ \text{Selberg's orthogonality conjecture}\ \Rightarrow\  \text{Conjecture~\ref{conjecture of SMO for unitary cuspidal automorphic repns}}.
\end{equation}

Motivated by \eqref{H implies SOC implies main conjecture}, we shall first prove certain results on Selberg's orthogonality conjecture under some mild conditions by showing that Hypothesis~\ref{Hypothesis implies Selberg orthogonality} holds for $m=5$. Then Conjecture~\ref{conjecture of SMO for unitary cuspidal automorphic repns} will follow immediately. More precisely, we have

\begin{proposition}\label{Strong MO for GL5 conditional}
Suppose that $m,m'\leq 5$. Let $\pi$ and $\pi'$ be self-dual unitary cuspdial automorphic representations of $\GL_m(\A)$ and $\GL_{m'}(\A)$, respectively. Assume that $L(s,\sym^2\pi)$ and $L(s,\sym^2\pi')$ are absolutely convergent for $\re(s)>1$. Then Selberg's orthogonality conjecture holds for the pair $(\pi,\pi')$. If we further assume that $\lambda_{\pi}(p)=\lambda_{\pi'}(p)$ for almost all primes $p$, then we have $m=m'$ and $\pi\simeq\pi'$.
\end{proposition}
In order to prove Proposition~\ref{Strong MO for GL5 conditional}, we need a technical lemma as follows. 
\begin{lemma}\label{technical lemma for n 5}
Let $z_1,z_2,\ldots,z_5$ be nonzero complex numbers satisfying the condition $|z_1|\geq |z_2|\geq\cdots\geq|z_5|$ and $z_i^{-1}=z_{6-i}, 1\leq i\leq 5$. Then we have
\begin{equation}\label{sect 3 a technical lemma for degree 5}
\sum_{i=1}^5|z_{i}|\leq 2\left|\sum_{i=1}^5z_i^2\right|^{1/2}+2\left|\sum_{i=1}^5z_i\right|+15.
\end{equation}
\end{lemma}
\begin{proof}
By the choice of $z_i$, we know that $|z_3|=1$ and $|z_4|,|z_5|\leq 1$. We first can show the following inequality
\begin{equation}\label{1st eq in the proof of technical lemma}
|z_1|+|z_2|\leq 2|z_1^2+z_2^2|^{1/2}+2|z_1+z_2|.
\end{equation}
To see \eqref{1st eq in the proof of technical lemma}, it suffices to show
\begin{equation}
|z_1|^2+|z_2|^2+2|z_1||z_2|\leq 4|z_1^2+z_2^2|+4|z_1+z_2|^2.
\end{equation}
Then it reduces to showing that
\begin{equation}\label{3rd eq in the proof of technical lemma}
-4(z_1\overline{z_2}+\overline{z_1}z_2)\leq 4|z_1^2+z_2^2|.
\end{equation}
In fact, let $z_1=r_1e^{i\theta_1}$ and $z_2=r_2e^{i\theta_2}$ with $r_1\geq r_2\geq 1$, and let $A=\theta_1-\theta_2$. A direct calculation implies that
\begin{equation}
-4(z_1\overline{z_2}+\overline{z_1}z_2)=8r_1r_2\cos A\leq 8r_1r_2|\cos A|\leq 4|z_1^2+z_2^2|
\end{equation}
as desired. Next, by \eqref{1st eq in the proof of technical lemma} we have
\begin{equation}\label{intermediate inequality}
\sum_{i=1}^5|z_{i}|\leq 2|z_1^2+z_2^2|^{1/2}+2|z_1+z_2|+3
\end{equation}
since $|z_3|,|z_4|,|z_5|\leq1$. On the other hand, we have
\begin{equation}\label{sqrt of sum of z1 and z2 squares}
    |z_1^2+z_2^2|^{1/2}\leq
    \begin{cases}
   \sqrt{3} & \mbox{if $|z_1^2+z_2^2|\leq 3$},\\    
  |z_1^2+z_2^2+z_3^2+z_4^2+z_5^2|^{1/2}+3&\mbox{otherwise}.
    \end{cases}
\end{equation}
Inserting \eqref{sqrt of sum of z1 and z2 squares} to \eqref{intermediate inequality}, we get
\[\sum_{i=1}^5|z_{i}|\leq2\left|\sum_{i=1}^5z_i^2\right|^{1/2}+2|z_1+z_2|+9\leq2\left|\sum_{i=1}^5z_i^2\right|^{1/2}+2\left(\left|\sum_{i=1}^5z_i\right|+|z_3+z_4+z_5|\right)+9.\]
The assertion \eqref{sect 3 a technical lemma for degree 5} follows immediately from the bounds $|z_3|,|z_4|,|z_5|\leq1$.
\end{proof}

Now we are ready to prove Proposition~\ref{Strong MO for GL5 conditional}.
\begin{proof}[Proof of Proposition~\ref{Strong MO for GL5 conditional}]Observing \eqref{H implies SOC implies main conjecture}, it suffices to prove Selberg's orthogonality conjecture for the pair $(\pi,\pi')$, which is further reduced to proving the Hypothesis~\ref{Hypothesis implies Selberg orthogonality}. Moreover, by Theorem~\ref{Thm Hypothesis H holds for m leq 4} (also see \cite[Corollary~1.5]{LiWaYe2005}) we just need to consider the case when $m,m'=5$. 

Since there are only finitely many ramified primes, it suffices to assume that the summation is over all unramified primes. Since $\pi$ is self-dual we have the Satake parameters $\set{\alpha_{p, j}}=\set{\overline{\alpha_{p, k}}}$ from \cite[(2.8)]{RudnickSarnak1996}. Then for each unramified prime $p$, we first rearrange the Satake parameters $\alpha_{p,1},\ldots,\alpha_{p,5}$ for $\pi$ such that $|\alpha_{p,1}|\geq\cdots\geq|\alpha_{p,5}|$ and for each $1\leq i\leq 5$, we have $|\alpha_{p, i}|=|\alpha_{p, 6-i}|^{-1}$. By Lemma~\ref{technical lemma for n 5} we have
\begin{equation}\label{fundamental inequality}
\sum_{i=1}^5|\alpha_{p,i}|\leq 2\left|\sum_{i=1}^5\alpha_{p,i}^2\right|^{1/2}+2\left|\sum_{i=1}^5\alpha_{p,i}\right|+15.
\end{equation}
It follows that
\begin{equation}\label{eq p ell square upper bound}
    a_{\pi}(p^{\ell})^2\ll \left|\sum_{i=1}^5\alpha_{p,i}^2\right|^{\ell}+\left|\sum_{i=1}^5\alpha_{p,i}\right|^{2\ell}+1.
\end{equation}
We also know that
\begin{equation}\label{GL5 Satake parameters square summation}
    \sum_{i=1}^5\alpha_{p,i}^2=2\lambda_{\sym^2\pi}(p)-\lambda_{\pi}(p)^2.
\end{equation}
Therefore, after plugging \eqref{GL5 Satake parameters square summation} into \eqref{eq p ell square upper bound} we obtain
\begin{equation}\label{Eqnapell2}
    a_{\pi}(p^{\ell})^2\ll |\lambda_{\sym^2\pi}(p)|^{\ell}+|\lambda_{\pi}(p)|^{2\ell}+1.
\end{equation}
By the trivial estimation of Satake parameters, we have
\begin{equation}\label{upper bounds for symmetric square and eigenvalue at p}
    |\lambda_{\sym^2\pi}(p)|\ll p^{1-2\delta}\quad\text{and}\quad |\lambda_{\pi}(p)|\ll p^{1/2-\delta}
\end{equation}
with $\delta=\delta_5=\frac{1}{26}$ by \cite[Theorem~2]{LuoRudnickSarnak1999}. In general, for $m\geq 2$ we have
\begin{equation}\label{EqnLRS99bound}
    \delta_m=\frac{1}{m^2+1}.
\end{equation}
Therefore, by \eqref{Eqnapell2} and \eqref{upper bounds for symmetric square and eigenvalue at p} we have
\begin{equation}\label{EqnProofofProp37Bound}
\begin{split}
\sum_{p}\frac{a_{\pi}(p^{\ell})^2\log^2p}{p^{\ell}}&\ll\sum_p\frac{|\lambda_{\sym^2\pi}(p)|^{\ell}\log^{2}p}{p^{\ell}}+\sum_p\frac{|\lambda_{\pi}(p)|^{2\ell}\log^{2}p}{p^{\ell}}+\sum_p\frac{\log^2p}{p^{\ell}}\\
&\ll\sum_p\frac{|\lambda_{\sym^2\pi}(p)|\log^{2}p}{p^{1+2\delta(\ell-1)}}+\sum_p\frac{|\lambda_{\pi}(p)|^{2}\log^{2}p}{p^{1+2\delta(\ell-1)}}+O(1)\\
&=O(1).
\end{split}
\end{equation}
Note that the last step is due to the assumption that $L(s,\sym^2\pi)$ is absolutely convergent for $\re(s)>1$ and the well-known fact that $L(s,\pi\otimes\pi)$ is absolutely convergent for $\re(s)>1$.
\end{proof}

Motivated by the proof of Proposition~\ref{Strong MO for GL5 conditional}, we can further have the following result for even more general cuspidal automorphic representations of arbitrary ranks under similar assumptions.

\begin{proposition}\label{Strong MO for GLn conditional}
Let $\pi$ and $\pi'$ be self-dual unitary cuspdial automorphic representations of $\GL_m(\A)$ and $\GL_{m'}(\A)$, respectively. Assume that for $r=2,3$, $L(s,\sym^r\pi)$ and $L(s,\sym^r\pi')$ are absolutely convergent for $\re(s)>1$. Then Selberg's orthogonality conjecture holds for the pair $(\pi,\pi')$. If we further assume that $\lambda_{\pi}(p)=\lambda_{\pi'}(p)$ for almost all primes $p$, then $m=m'$ and $\pi\simeq\pi'$.
\end{proposition}
\begin{proof}
Similar with the proof of Proposition~\ref{Strong MO for GL5 conditional}, it suffices to show the Hypothesis~\ref{Hypothesis implies Selberg orthogonality} for $\pi$ and $\pi'$. Here, we only consider $\pi$ and analogous arguments work for $\pi'$ as well. By assumption, $L(s,\sym^3\pi)$ is absolutely convergent for $\re(s)>1$. This implies that for all unramified primes $p$, we have $|\alpha_{p,i}|<p^{1/3+\epsilon}$ for any $\epsilon>0$ and $1\leq i\leq m$. In particular, we can take $\epsilon=\frac{1}{48}$. It follows that for $\ell\geq4$,
\begin{equation}
\sum_{p}\frac{a_{\pi}(p^{\ell})^2\log^2p}{p^{\ell}}<\infty
\end{equation}
by the trivial estimation. So it suffices to consider $\ell=2,3$. On the other hand, for $r=2,3$ we set
\begin{equation}
L(s,\sym^r\pi)=\sum_{n\geq 1}\frac{\lambda_{\sym^r\pi}(n)}{n^s}
\end{equation}
when $\re(s)>1$. By the theory of symmetric polynomials, we have
\begin{align}
    a_{\pi}(p^2)=&2\lambda_{\sym^2\pi}(p)-\lambda_{\pi}(p)a_{\pi}(p),\label{symmetric square relation}\\
  a_{\pi}(p^3)=&3\lambda_{\sym^3\pi}(p)-\lambda_{\sym^2\pi}(p)a_{\pi}(p)-\lambda_{\pi}(p)a_{\pi}(p^2).\label{symmetric cubic relation}  
\end{align}
By \cite[Theorem~2]{LuoRudnickSarnak1999}, we obtain
\begin{equation}\label{eq.LRS corollary}
a_{\pi}(p^{\ell})\ll p^{\ell\left(\frac{1}2-\delta\right)},
\end{equation}
where $\delta=\delta_{m}=\frac{1}{m^2+1}$ (see \eqref{EqnLRS99bound}) and $m$ is the degree of the $L$-function $L(s,\pi)$. 

From \eqref{Eqnapipell} we have $a_{\pi}(p) = \lambda_{\pi}(p)$. Plugging \eqref{symmetric square relation} into \eqref{pi H} and applying the bound \eqref{eq.LRS corollary} with $\ell = 2$ yield
\begin{align*}
    \sum_{p}\frac{a_{\pi}(p^{2})^2\log^2p}{p^{2}}\ll &\sum_{p}\frac{\abs{2\lambda_{\sym^2\pi}(p)-\lambda_{\pi}(p)^2}\log^2p}{p^{1+2\delta}}\\
    \ll& \sum_p\frac{\abs{\lambda_{\sym^2\pi}(p)}\log^2p}{p^{1+2\delta}}+\sum_p\frac{\abs{\lambda_\pi(p)}^2\log^2p}{p^{1+2\delta}}\\
    =&O(1).
\end{align*}
Again, the last step is due to the assumption that $L(s,\sym^2\pi)$ is absolutely convergent for $\re(s)>1$ and the well-known fact that $L(s,\pi\otimes\pi)$ is absolutely convergent for $\re(s)>1$. 

Similarly, plugging \eqref{symmetric cubic relation} into \eqref{pi H} and applying \eqref{eq.LRS corollary} (with $\ell = 3$), \eqref{symmetric square relation} and \eqref{upper bounds for symmetric square and eigenvalue at p} yield
\begin{align*}
    \sum_{p}\frac{a_{\pi}(p^{3})^2\log^2p}{p^{3}}\ll&\sum_{p}\frac{\abs{3\lambda_{\sym^3\pi}(p)-\lambda_{\sym^2\pi}(p)a_{\pi}(p)-\lambda_{\pi}(p)a_{\pi}(p^2)}\log^2p}{p^{\frac{3}{2}+3\delta}}\\
    \ll& \sum_p\frac{\abs{\lambda_{\sym^3\pi}(p)}\log^2p}{p^{\frac{3}{2}+3\delta}}+\frac{\abs{\lambda_{\sym^2\pi}(p)}\log^2p}{p^{1+4\delta}}+\sum_p\frac{\abs{\lambda_\pi(p)}^2\log^2p}{p^{1+4\delta}}\\
    =&O(1).
\end{align*}
This now follows from the assumption that $L(s,\sym^2\pi)$ and $L(s,\sym^3\pi)$ are absolutely convergent for $\re(s) > 1$, together with the well-known fact that $L(s,\pi \otimes \pi)$ also converges absolutely in this region. This completes the proof.
\end{proof}

We conclude this subsection by summarizing the above results on Conjecture~\ref{conjecture of SMO for unitary cuspidal automorphic repns}.
\begin{theorem}\label{SMO for GLn by L series}
Let $\pi$ and $\pi'$ be self-dual unitary cuspdial automorphic  representations of $\GL_m(\A)$ and $\GL_{m'}(\A)$, respectively. If one of the following conditions holds, then Conjecture~\ref{conjecture of SMO for unitary cuspidal automorphic repns} is true:
\begin{enumerate}[(i)]
    \item $m,m'\leq4$. 
    \item $m,m'=5$ and the symmetric square $L$-functions $L(s,\sym^2\pi)$ and $L(s,\sym^2\pi')$ are absolutely convergent for $\re(s)>1$.
    \item For general $m, m'$, the symmetric square $L$-functions $L(s,\sym^2\pi)$, $L(s,\sym^2\pi')$ and the symmetric cube $L$-functions  $L(s,\sym^3\pi)$, $L(s,\sym^3\pi')$ are absolutely convergent for $\re(s)>1$.
    \item
    The Generalized Ramanujan Conjecture holds for at least one of $\pi$ and $\pi^\prime$.
    \item A partial Ramanujan bound for some $\theta<1/4$ holds for both $\pi$ and $\pi'$.
\end{enumerate}
\end{theorem}
\begin{proof}
For (i), see Theorem~\ref{Thm Hypothesis H holds for m leq 4} and \eqref{H implies SOC implies main conjecture}. For (ii)-(iii), see Proposition~\ref{Strong MO for GL5 conditional} and Proposition~\ref{Strong MO for GLn conditional}. For (iv)-(v), see Proposition~\ref{SMO for GLn by L series under GRC}.
\end{proof}

\subsection{Distinguishing paramodular newforms in terms of the spinor \texorpdfstring{$L$}{L}-functions}\label{section distinguish by spinor}
In this subsection, we consider an application of strong multiplicity one results of $\GL_m$ (see Theorem~\ref{SMO for GLn by L series}) to distinguishing paramodular newforms. Let $F\in\mathcal{S}_{k, j}^{\mathrm{new}}(K(N))$ be a paramodular newform of level $N$ and weight $(k, j)\in \Z_{>0}\times\Z_{\geq 0}$. Then we can associate with $F$ a cuspidal automorphic representation $\pi_F$ of $\GSp_4(\A)$ with trivial central character. Moreover, we can define the corresponding spinor $L$-function $L(s, \pi_F, \rho_4)=\sum_{n\geq1}a_{F}(n)n^{-s}$ as in \eqref{finite part of spinor L function}.

If we further assume that $F$ is of type {\bf (G)}, then  we can find a self-dual, symplectic, unitary, cuspidal automorphic representation $\Pi_{4, F}$ of $\GL_4(\A)$ such that $L(s,\pi_F,\rho_4)=L(s,\Pi_{4, F})$; see for example \cite{Schmidt2018}. The following result indicates that we can distinguish paramodular newforms using only $a_F(p)$ for almost all primes $p$.

\begin{theorem}\label{main thm 1}
Let $(k_i, j_i)\in\Z_{>0}\times\Z_{\geq 0}$ and $N_i\in\Z_{>0}$, and let $F_i\in \mathcal{S}_{k_i, j_i}^{\mathrm{new}}(K(N_i))$ be two Hecke eigenforms for $i=1, 2$. For $\re(s)>3/2$, let 
$$
L(s,\pi_{F_i},\rho_4)=\sum_{n\geq 1}\frac{a_{F_i}(n)}{n^s},\quad i=1, 2,
$$
be the associated spinor $L$-functions. If $a_{F_1}(p)=a_{F_2}(p)$ for almost all primes $p$, then $F_1=c\cdot F_2$ for some nonzero constant $c$.
\end{theorem}
\begin{proof}
Observing \eqref{para decomp to G and P types} and \eqref{para decomp to G and P types only G}, only the types {\bf (G)} and {\bf (P)} could happen in paramodular forms, and only the type {\bf (G)} shows up if we consider the vector-valued paramodular forms. Thus, we will separate three cases as follows.

i) We first assume that they are both of type {\bf (G)}, then by \cite{Schmidt2018} we can find $\Pi_{F_i}$, self-dual, symplectic, unitary, cuspidal automorphic representations of $\GL_4(\A)$ such that
\begin{equation}
    L(s,\pi_{F_i},\rho_4)=L(s,\Pi_{F_i}),\quad i=1, 2.
\end{equation}
Recall that we only consider the finite part of completed $L$-functions, i.e., there is no archimedean factor. It follows from Theorem~\ref{SMO for GLn by L series} (i) that $\Pi_{F_1}\simeq\Pi_{F_2}$ and hence $L(s,\pi_{F_1},\rho_4)=L(s,\pi_{F_2},\rho_4)$. Then the desired assertion in this case immediately follows from \cite[Theorem~2.6]{Schmidt2018}. 

ii) Assume that they are both of type {\bf (P)}, i.e., \textit{Gritsenko} or \textit{Saito-Kurokawa liftings}. Moreover, in this case we have $j_1=j_2=0$ as indicated above. Then we have eigenforms $f_i\in \mathcal{S}_{2k_i-2}^{\text{new}}(\Gamma_0^{(1)}(N_i))$ corresponding to $F_i, i=1, 2$. Note that the sign in the functional equations of $L(s, f_i)$ both are $-1$. Let $S$ be the set that contains all primes $p\mid \mathrm{lcm}[N_1, N_2]$, i.e., all ramified primes. And for all $p\notin S$, we have $a_{F_1}(p)=a_{F_2}(p)$ by assumption. As we also know that for $i=1, 2$,
\begin{equation}\label{SK lift L function in terms of f and zeta}
    L^S(s, \pi_{F_i}, \rho_4)=\zeta(s+1/2)\zeta(s-1/2)L^S(s, f_i),
\end{equation}
where 
\begin{equation}
    L(s, f_i)=\sum_{n\geq 1}\frac{\lambda_{f_i}(n)}{n^s}.
\end{equation}
Here, $\lambda_{f_i}(n)$ are the Hecke eigenvalues associated to $f_i$. Then for any $p\notin S$, we have 
\begin{equation}\label{Saito-Kurokawa primes}
    a_{F_i}(p)=p^{1/2}+p^{-1/2}+\lambda_{f_i}(p).
\end{equation}
In particular, this shows that $f_1$ is a constant multiple of $f_2$. Then by \cite[Theorem~5.5]{Schmidt2020}, we obtain $F_{f_1}=c\cdot F_{f_2}$ for some nonzero constant $c$ as desired.

iii) If we assume that $F_1$ is of type {\bf (P)} and $F_2$ is of type {\bf (G)}, then we can obtain a contradiction as follows. In fact, by \eqref{Saito-Kurokawa primes} and \eqref{EqnLRS99bound}, we know that
\begin{equation}
|a_{F_1}(p)|\geq \frac{1}{2}p^{1/2}
\end{equation}
for all sufficiently large $p$. On the other hand, since $F_2$ is of type {\bf (G)} by assumption, we have
\begin{equation}
|a_{F_2}(p)|\leq 4p^{1/2-\delta},
\end{equation}
for $\delta=\delta_4=\frac{1}{17}$ by \eqref{EqnLRS99bound}. Therefore, $a_{F_1}(p)\neq a_{F_2}(p)$ for any sufficiently large $p$, which is the desired contradiction.
\end{proof}

\section{Strong multiplicity one in terms of the standard \texorpdfstring{$L$}{L}-functions}
\label{section-standard}

In this section, we aim to give a result of distinguishing paramodular newforms by using the standard $L$-functions. More precisely, we will first prove the cuspidality of the automorphic representation associated to the standard $L$-function for a Hecke eigenform of type {\bf (G)}; see Theorem~\ref{Thm Pi5 cuspidal unitary} below. As a consequence, we will give a result about distinguishing paramodular cusp forms in terms of such $L$-functions via the Galois representation theoretic approach; see Theorem~\ref{main thm 2} below.

\subsection{Cuspidality of automorphic representations associated to the standard \texorpdfstring{$L$}{L}-functions}\label{cuspidality of GL5}

Let $F\in\mathcal{S}_{k, j}^{\mathrm{new}, \text{{\bf (G)}}}(K(N))$ be a paramodular newform of level $N$ and weight $(k, j)\in \Z_{>0}\times\Z_{\geq 0}$; see \eqref{para decomp to G and P types}. Then we can associate with $F$ a cuspidal automorphic representation $\pi_F$ of $\GSp_4(\A)$, which can be transferred to a self-dual, symplectic, unitary, cuspidal automorphic representation $\Pi_{4,F}$ of $\GL_4(\A)$. If we further assume that the local representation of $\Pi_{4,F}$ at primes $p=2,3$ are not supercuspidal, then by \cite[Theorem~A]{Kim2003} the exterior square $\Lambda^2\Pi_{4,F}$ is an automorphic representation of $\GL_6(\A)$. Since $\Pi_{4,F}$ is symplectic (i.e., its exterior square $L$-function has a pole at $s=1$), we can write
\begin{equation}\label{decomp of exterior square of Pi 4 F}
\Lambda^2\Pi_{4,F}\simeq\mathbf{1}\boxplus\Pi_{5,F}.
\end{equation}
Here, $\mathbf{1}$ is the trivial representation. A direct calculation indicates that the $L$-function of $\Pi_{5,F}$ in fact coincides with the standard $L$-function of $\pi_F$; see \eqref{ext rho 4 equals 1 plus rho 5}. Before giving a cuspidality criterion for $\Pi_{5,F}$ as in \eqref{decomp of exterior square of Pi 4 F}, we will need the following lemma.

\begin{lemma}\label{smo for twist cases}
Let $F_i\in\mathcal{S}_{k_i, j_i}^{\mathrm{new}, \text{{\bf (G)}}}(K(N_i)), i=1, 2$, be paramodular newforms of level $N_i$ and weight $(k_i, j_i)\in \Z_{>0}\times\Z_{\geq 0}$. With the notation as above, suppose that there exists a finite order character $\chi$ such that $\Pi_{4,F_1}\simeq\Pi_{4,F_2}\otimes\chi$ as cuspidal automorphic representations of $\GL_4(\A)$. Then $\pi_{F_1}\simeq\pi_{F_2}\otimes\chi$ as cuspidal automorphic representations of $\GSp_4(\A)$. 
\end{lemma}
\begin{proof}
Notice that $\Pi_{4,F_i}, i=1, 2$, have the trivial central character, and then $\chi$ is a quartic character by assumption. Then by the definition of Arthur packets, $\pi_{F_1}$ and $\pi_{F_2}\otimes\chi$ are in the same global Arthur packet and hence the same local Arthur packet for each local place. We write $\pi_{F_i}=\pi_{F_i,\infty}\otimes(\otimes_{p<\infty}^{'}\pi_{F_i,p})$ for $i=1, 2$, and $\chi=\chi_{\infty}\otimes(\otimes_{p<\infty}^{'}\chi_p)$.

Since $F_1,F_2$ are paramodular newforms of type \text{{\bf (G)}}, by \cite[Theorem~1.1]{Schmidt2018} the local representation at each non-archimedean place is generic. Moreover, such generic local representation is unique in each local Arthur packet and hence $\pi_{F_1,p}\simeq\pi_{F_2,p}\otimes\chi_p$ for all finite places $p$. As for the archimedean place, the local representation is non-generic since the paramodular forms under consideration are holomorphic. A result of Schmidt (\cite[Lemma~2.3]{Schmidt2018}) shows that in each paramodular type at archimedean place we have $k\geq 2$. Notice that there are exactly two local representations in each packet at the archimedean place: one is generic and the other one is non-generic; see for example \cite[\S~3]{Schmidt2017}. This implies that $\pi_{F_1,\infty}\simeq\pi_{F_2,\infty}\otimes\chi_{\infty}$ since they are both non-generic. By \cite[Theorem~2.6]{Schmidt2018}, we have $\pi_{F_1}\simeq\pi_{F_2}\otimes\chi$ as cuspidal automorphic representations of $\GSp_4(\A)$. 
\end{proof}

We are ready to show the cuspidality of $\Pi_{5, F}$ as in \eqref{decomp of exterior square of Pi 4 F} under some mild condition as follows. Note that the full level case has been done in \cite[Theorem~5.1.5]{PitaleSahaSchmidt2014}.

\begin{theorem}\label{Thm Pi5 cuspidal unitary}
Let $\pi_F$ be a self-dual cuspidal automorphic representation of $\GSp_4(\A)$ associated to a paramodular newform $F$ of type {\bf (G)}. Suppose that the local components $\pi_{F,2}$ and $\pi_{F,3}$ of $\pi_F$ are not supercuspidals and there does not exist non-trivial quartic character $\chi$ such that $\pi_F\simeq\pi_F\otimes\chi$, then the automorphic representation $\Pi_{5, F}$ of $\GL_5(\A)$ corresponding to $\pi_F$ is cuspidal. Moreover, such $\Pi_{5, F}$ is unitary.
\end{theorem}
\begin{proof}
We prove this by contradiction.  Let $S$ be the set consisting of ramified primes of $\Pi_{5, F}$ and primes $p=2,3$. If the hypothesis
of the Theorem is not true, then the partial standard $L$-function
\begin{equation}
L^S(s,\Pi_{5, F})=L^S(s, \pi_F, \rho_5)=L^S(s,\tau_1)^{m(\tau_1)}\cdots L^S(s,\tau_r)^{m(\tau_r)},
\end{equation}
where $\tau_i$ is a cuspidal automorphic representation of $\GL_{n_i}(\A)$, $\tau_i$ distinct and $\sum_{i=1}^r n_im(\tau_i)=5$. By the assumption in \cite[Page~139]{Kim2003}, we know that all $\tau_i$ are unitary representations. Furthermore, we can assume that $n_1\leq n_2\leq\cdots \leq n_r$. 

\textbf{Case I:} If $n_1=1$, then $\tau_1=\chi$ for some character and hence $L^S(s,\Pi_{5, F}\times\chi^{-1})$ has a pole at $s=1$. On the other hand, we can show that the partial exterior square $L$-function of $\Pi_{5, F}$, denoted by $L^S(s, \Pi_{5, F}, \Lambda^2)$, is nonzero at $s=1$. Indeed, for $n_2\neq4$, i.e., $n_i\leq 3$ with $1\leq i\leq r$, we know that $\Lambda^2\Pi_{5, F}$ is automorphic and hence $L^S(s, \Pi_{5, F}, \Lambda^2)\neq 0$ at $s=1$ as desired; and when $n_2=4$, we have $\Pi_{5, F}\simeq\chi\boxplus\tau_2$ and hence 
\begin{equation}
L^S(s,\Pi_{5, F}, \Lambda^2)=L^S(s,\chi\boxplus\tau_2,\Lambda^2)=L^S(s,\chi\otimes\tau_2)L^S(s,\tau_2,\Lambda^2).
\end{equation}
By \cite[Theorem~5.3.1]{Kim2003}, we can find an automorphic representation $\widetilde{\Pi}$ of $\GL_6(\A)$ such that 
\begin{equation}
L^S(s,\tau_2,\Lambda^2)=L^S(s,\widetilde{\Pi})
\end{equation}
and hence $L^S(s,\tau_2,\Lambda^2)\neq 0$ at $s=1$ as well. 

Next, recall that $\Pi_{4, F}$ is a self-dual, symplectic, unitary, cuspidal automorphic representation $\Pi_{4,F}$ of $\GL_4(\A)$, which is a strong functorial lifting of $\pi_F$. Therefore, for $\tau_1=\chi$, observing \eqref{ext 5 equal sym 4} we have
\begin{equation}\label{nonvanishing of sym 4}
L^S(1,\Pi_{4, F},\sym^2\times\chi^{-1})=L^S(1, \Pi_{5, F}, \Lambda^2\times \chi^{-1})\neq 0.
\end{equation}
Then it follows from \eqref{decomp of exterior square of Pi 4 F} and \eqref{nonvanishing of sym 4} that
\begin{equation}
L^S(s,\Pi_{4, F}\times\Pi_{4, F}\times\chi^{-1})=L^S(s,\chi^{-1})L^S(s,\Pi_{5, F}\times\chi^{-1})L^S(s,\Pi_{4, F},\sym^2\times\chi^{-1})
\end{equation}
has a pole at $s=1$. In addition, we know that the product of any finite number of local $L$-functions is non-vanishing at $s=1$. This implies that 
\begin{equation}
 L(s,\Pi_{4, F}\times\Pi_{4, F}\times\chi^{-1})=\left(\prod_{p\in S}L_p(s,\Pi_{4, F}\times\Pi_{4, F}\times\chi^{-1})\right)L^S(s,\Pi_{4, F}\times\Pi_{4, F}\times\chi^{-1})
\end{equation} 
has a pole at $s=1$. Hence 
\begin{equation}\label{pi 4 equals pi 4 chi}
\Pi_{4, F}\simeq\Pi_{4, F}\otimes\chi
\end{equation}
since $\Pi_{4, F}$ is self-dual. Moreover, we can see that $\chi$ is a quartic character, i.e., $\chi^4=\mathbf{1}$ by taking the central character in both sides of \eqref{pi 4 equals pi 4 chi}. Then by Lemma~\ref{smo for twist cases}, we see that $\pi_F\simeq\pi_F\otimes\chi$, a contradiction.

\textbf{Case II:} Next we assume that 
\begin{equation}
L^S(s,\Pi_{5, F})=L^S(s,\tau_1)L^S(s,\tau_2),
\end{equation}
where $\tau_1$ is a unitary cuspidal automorphic representation of $\GL_{2}(\A)$ and $\tau_2$ is a unitary cuspidal automorphic representation of $\GL_{3}(\A)$. Then by a straightforward calculation we have 
\begin{equation}\label{case II ext 5 decomp}
L^S(s,\Pi_{5, F},\Lambda^2)=L^S(s,\omega_{\tau_1})L^S(s,\tau_1\times\tau_2)L^S(s,\omega_{\tau_2}\widetilde{\tau_2}).
\end{equation}
Here, $\omega_{\tau_i}$ is the central character of $\tau_i, i=1, 2$, and $\widetilde{\tau_2}$ is the contragredient representation of $\tau_2$. Then by \eqref{ext 5 equal sym 4} and \eqref{case II ext 5 decomp}, we see that
\begin{equation}
L^S(s,\Pi_{4, F},\sym^2\times\omega_{\tau_1}^{-1})=L^S(s,\Pi_{5, F},\Lambda^2\times\omega_{\tau_1}^{-1})
\end{equation}
has a simple pole at $s=1$. Similarly as the previous case, we can obtain that
\begin{equation}
L^S(s,\Pi_{4, F}\times\Pi_{4, F}\times\omega_{\tau_1}^{-1})=L^S(s,\omega_{\tau_1}^{-1})L^S(s,\Pi_{5, F}\times\omega_{\tau_1}^{-1})L^S(s,\Pi_{4, F},\sym^2\times\omega_{\tau_1}^{-1})
\end{equation}
has a simple pole at $s=1$. And so
\begin{equation}
L(s,\Pi_{4, F}\times\Pi_{4, F}\times\omega_{\tau_1}^{-1})=L(s,\omega_{\tau_1}^{-1})L(s,\Pi_{5, F}\times\omega_{\tau_1}^{-1})L(s,\Pi_{4, F},\sym^2\times\omega_{\tau_1}^{-1})
\end{equation}
has a simple pole at $s=1$. A similar argument implies that $\pi_F\simeq\pi_F\otimes\omega_{\tau_1}^{-1}$ with $\omega_{\tau_1}^{-4}=\mathbf{1}$, a contradiction. 

In conclusion, we have shown that $\Pi_{5, F}$ is cuspidal. Again, observing the assumption in \cite[Page~139]{Kim2003}, $\Pi_{5, F}$ is a unitary representation.
\end{proof}

Motivated by Theorem~\ref{Thm Pi5 cuspidal unitary}, we introduce the following definition for a certain class of Hecke eigenforms.

\begin{definition}\label{definition of good}
A Hecke eigenform $F\in \mathcal{S}_{k, j}(K(N))$ is called \textit{good} if the local components $\pi_{F,2}$ and $\pi_{F,3}$ of the associated cuspidal automorphic representation $\pi_F$ of $\GSp_4(\A)$ are not supercuspidals and there does not exist a non-trivial character $\chi$ such that $\pi_F\simeq\pi_F\otimes\chi$.
\end{definition}
As mentioned in Remark~\ref{rk 2 after main thm 2}, the conditions on the local representations at primes $p=2, 3$ in the definition of a Hecke eigenform $F$ being \textit{good} comes from the assumptions in \cite[Theorem~A]{Kim2003}.

\subsection{Distinguishing paramodular newforms in terms of the standard \texorpdfstring{$L$}{L}-functions}
Before we prove the main result in this subsection, we first review some standard facts concerning Galois representations attached to Siegel modular forms of degree $2$. 

Let $(k, j)\in\Z_{\geq 3}\times\Z_{\geq 0}$. Let $\pi$ be a cuspidal automorphic representation of $\GSp_4(\mathbb{A})$ with trivial central character of weight $(k, j)$. Let $S$ be the set of places of $\Q$ at which $\pi$ is ramified. For every rational prime $\ell$, one can attach a semisimple symplectic Galois representation
\begin{equation}
r_{\pi , \ell}\colon \Gal(\overline{\Q}/\Q) \rightarrow \GSp_4(\overline{\Q}_{\ell})
\end{equation}
that satisfying the following properties:
\begin{enumerate}
    \item[1.] the similitude character of $r_{\pi , \ell}$ is $\mathrm{sim}\,r_{\pi , \ell}=\varepsilon_{\ell}^{-(2k + j -3)}$, where $\varepsilon_{\ell}$ is the $\ell$-adic cyclotomic character;
    \item[2.] $r_{\pi , \ell}^{\vee} \cong r_{\pi , \ell}\otimes \varepsilon_{\ell}^{2k + j -3}$ and $\det\,r_{\pi , \ell}= \mathrm{sim}^2r_{\pi, \ell}= \varepsilon_{\ell}^{-2(2k + j -3)}$;
    \item[3.] $r_{\pi, \ell}$ is unramified at all finite places $v$ such that $v\not \in S\cup \{\ell\}$;
    \item[4.] $r_{\pi, \ell}$ is de Rham at the place $\ell$ and crystalline if $\ell\not\in S$; 
    \item[5.] $r_{\pi, \ell}$ is irreducible for all but finitely many prime $\ell$; and
    \item[6.] the Hodge-Tate weights of $r_{\pi, \ell}$ is $\{0,k-2, k+j-1 ,2k +j -3\}$.
\end{enumerate}
See \cite{Tay1993, Lau2005, Wei2005} for references and more general results.

Suppose that $F\in \mathcal{S}_{k, j}(K(N))$ is a paramodular cusp form. Then we know that the associated cuspidal automorphic representation $\pi_F$ of $\GSp_4(\mathbb{A})$ has the minimal $K$-type $(k+j, k)$ and trivial central character. 
Furthermore, the Hodge-Tate weights of $r_{\pi_F, \ell}$ are $\{0, k-2, k+j-1, 2k+j-3\}$. Then $r_{\pi_F, \ell}$ is regular (i.e., has distinct Hodge-Tate weights) when $k\geq 3$. The main theorem in this subsection is as follows.

\begin{theorem}\label{main thm 2}
Let $(k_i, j_i)\in\Z_{\geq 3}\times\Z_{\geq 0}$ and $N_i\in\Z_{>0}$ for $i=1, 2$. Let $F_i\in \mathcal{S}_{k_i, j_i}^{\mathrm{new}}(K(N_i)), i=1, 2$, be two Hecke eigenforms, which are \textit{good} as defined in Definition~\ref{definition of good}.
For $\re(s)>3/2$, let 
$$
L(s,\pi_{F_i},\rho_5)=\sum_{n\geq 1}\frac{b_{F_i}(n)}{n^s},\quad i=1, 2,
$$
be the associated standard $L$-functions. If $b_{F_1}(p)=b_{F_2}(p)$ for almost all primes $p$, then there exists a quadratic character $\chi$ such that $\pi_{F_1}\simeq\pi_{F_2}\otimes\chi$. Additionally, if we assume that $N_1$ and $N_2$ are squarefree numbers, then $F_1=c\cdot F_2$ for some nonzero constant $c$.
\end{theorem}

\begin{proof}
Again, observing \eqref{para decomp to G and P types} and \eqref{para decomp to G and P types only G}, only the types {\bf (G)} and  {\bf (P)} could happen in paramodular forms and only the type {\bf (G)} shows up if we consider the vector-valued paramodular forms. This implies that the proof for the vector-valued case is actually even easier than that for the scalar-valued case. Thus, we will focus on the scalar-valued paramodular forms and hence we will consider three cases similarly as in the proof of Theorem~\ref{main thm 1}.

i) Suppose that both of $F_i, i=1, 2$, are of type {\bf (P)}. Then $F_i$ is a Saito-Kurokawa lifting of some elliptic Hecke eigenform $f_i\in \mathcal{S}_{2k_i-2}^{\text{new}}(\Gamma_0^{(1)}(N_i))$ for $i=1, 2$. Let $\lambda_{f_{i}}(p)$ be the Hecke eigenvalue of $f_{i}$ at prime $p\nmid \mathrm{lcm}[N_1, N_2]$. Then a direct calculation shows that
\begin{equation}\label{bF equals 1 Hecke eigenvalue p pm half}
b_{F_i}(p)=1+\lambda_{f_i}(p)(p^{1/2}+p^{-1/2}).
\end{equation}
Then, in this case, the desired result is an immediate consequence of the strong multiplicity one theorem for $\GL_2$.

ii) Without lose of generality, suppose that $F_1$ is of type {\bf (P)} and  $F_2$ is of type {\bf (G)}. Then we know that $F_1$ is a Saito-Kurokawa lifting of some $f_1\in \mathcal{S}_{2k_1-2}^{\text{new}}(\Gamma_0^{(1)}(N_1))$. By the Sato-Tate theorem \cite{HST2010}, we can always find a sequence of primes $\{p\}$ such that $|\lambda_{f_1}(p)|\geq \delta_0>0$. In particular, it follows from \eqref{bF equals 1 Hecke eigenvalue p pm half} that we find a sequence of primes $\{p\}$ such that
\begin{equation}
b_{F_1}(p)\gg p^{1/2}.
\end{equation}
On the other hand, it follows from Theorem~\ref{Thm Pi5 cuspidal unitary} we can associate $F_2$ with a cuspidal automorphic representation of $\GL_5(\A)$ due to the assumption that $F_2$ is \textit{good} of type {\bf (G)}. Then by \eqref{EqnLRS99bound} we can see that
\begin{equation}
b_{F_2}(p)\leq 5p^{1/2-1/26},
\end{equation}
which is a contradiction. That is, this case will never happen.

iii) Finally, we assume that both of $F_i, i=1, 2$, are of type {\bf (G)}. Let 
\begin{equation}
r_i\colon \Gal(\overline{\Q}/\Q) \rightarrow \GSp_4(\overline{\Q}_{\ell}), \quad i=1, 2,
\end{equation}
be the regular geometric Galois representation attached to $\pi_{F_i}$. We denote the projective representations of $r_i$ by
\begin{equation}
\mathrm{proj}\,r_i \colon \Gal(\overline{\Q}/\Q)\rightarrow \PGSp_4(\overline{\Q}_{\ell}).
\end{equation}
Then
\begin{equation}
\wedge^2 r_{i} \otimes \mathrm{sim}^{-1}r_{i} \cong \mathrm{std}\,r_i  \oplus \mathbf{1}, \quad i=1, 2,
\end{equation}
where $\mathrm{std}\,r_i$ can be identified with the projective representation $\mathrm{proj}\,r_i$ via the exceptional isomorphism of the algebraic group $\PGSp_4$ with the split orthogonal group $\SO_5$.
Next, by the assumption that $b_{F_1}(p) = b_{F_2}(p)$ for almost all primes $p$, together with the fact that $b_{F_i}(p) =\trace \,\mathrm{proj}\,r_{i}(\Frob_{p})$, we obtain
\begin{equation}
\trace \,\mathrm{proj}\,r_{1}(\Frob_{p})= \trace\, \mathrm{proj}\,r_{2}(\Frob_p).
\end{equation}
Here, $\Frob_p$ denotes the Frobenius element at $p$, and $\trace \,\mathrm{proj}\,r_{i}(\Frob_{p})$ is the trace of the endomorphism $\mathrm{proj}\,r_{i}(\Frob_{p})$. By Theorem~\ref{Thm Pi5 cuspidal unitary} and \cite[Theorem~3.2]{CG-irred}, $\mathrm{proj}\,r_{i}$ are irreducible as $\GL_5$-representations.
By the Brauer-Nesbitt theorem, $\mathrm{proj}\,r_1 $ is isomorphic to $\mathrm{proj}\,r_2$ as $\GL_5$-representations. Moreover, they are isomorphic to each other as $\SO_5$-representations (see \cite[Appendix]{Ram2000} by Ramakrishnan). So $\mathrm{proj}\,r_{1} \cong \mathrm{proj}\,r_{2}$ as $\PGSp_4$-representations. This implies that $r_1\cong\varphi r_2$
for some geometric character $\varphi$. In particular, $\trace\,r_1(\Frob_p) = \varphi(\Frob_p)\trace\,r_2(\Frob_p)$ for almost all primes $p$.

Also note that the Hodge-Tate weights of $\mathrm{proj}\,r_{i}$ are 
\begin{equation}
\{k_i + j_i -1, k_i -2, 0, -k_i +2, -k_i -j_i +1\},\quad i=1, 2.
\end{equation}
Then the isomorphism between $\mathrm{proj}\,r_{1} $ and $ \mathrm{proj}\,r_2$ implies that $k_1 = k_2$ and $j_1= j_2$. So the Hodge-Tate weights of $r_{1}$ and $r_2$ are the same. This implies that the Hodge-Tate weight of $\varphi$ is $0$. By the classification of one-dimensional geometric Galois representations, it follows that $\varphi$ is a character of finite order. 

Let $\chi$ be the Hecke character of $\Q^\times\backslash\A^\times$ associated to $\varphi$. Since 
$$\trace\,r_1(\Frob_p) = \varphi(\Frob_p)\trace\,r_2(\Frob_p)$$
for almost all primes $p$, we obtain that 
\begin{equation}
    \lambda_{\pi_{F_1}}(p) = \lambda_{\chi}(p) \lambda_{\pi_{F_2}}(p) = \lambda_{\pi_{F_2}\otimes \chi}(p)
\end{equation}
for almost all primes $p$. By Theorem~\ref{Thm Hypothesis H holds for m leq 4} for $m=4$ and \eqref{H implies SOC implies main conjecture} (or Theorem~\ref{SMO for GLn by L series} (i)) the associated cuspidal automorphic representations of $\GL_4(\A)$ corresponding to $F_1$ and $F_2$ differ by a finite order character $\chi$. Then by Lemma~\ref{smo for twist cases} we obtain $\pi_{F_1}\simeq\pi_{F_2}\otimes\chi$ as claimed.

Then, to see that $\chi$ is indeed a quadratic character, we first have that $\chi$ is a quartic character since $\pi_{F_1}$ and $\pi_{F_2}$ have trivial central character. Moreover, by $\pi_{F_1}\simeq\pi_{F_2}\otimes\chi$, one can see that
\begin{equation}
    L(s,\pi_{F_1}\times\pi_{F_1})=L(s,\pi_{F_2}\times \pi_{F_2}\chi^2)
\end{equation}
has a pole at $s=1$. In particular, this implies that $\pi_{F_2}\simeq\pi_{F_2}\otimes\chi^2$ as cuspidal automorphic representations of $\PGSp_4(\A_{\Q})$. Since $F_2$ is good by assumption, there does not exist non-trivial character $\chi$ such that $\pi_{F_2}\simeq\pi_{F_2}\otimes\chi$. It follows that $\chi^2$ must be the trivial character and hence $\chi$ is a quadratic character as expected.

As for the last assertion, first we know that the local components of cuspidal automorphic representations $\pi_{F_i}$ of $\GSp_4(\A)$ at primes $p\mid N_i, i=1, 2$, are Iwahori-spherical representations (see \cite{RobertsSchmidt2007}) due to the condition that the levels $N_i$ are squarefree. Hence $\chi$ is unramified for all primes $p$. This implies that $\chi=|\cdot|_{\A}^{it}$; see for example \cite[Proposition~3.1.2]{Bump1997}. And it is easy to see that $\chi$ is self-dual. Hence $t=0$, i.e., $\chi=\mathbf{1}$, the trivial representation as desired.
\end{proof}

\section{Distinguishing eigenforms by the central values of the twisted spinor \texorpdfstring{$L$}{L}-functions}\label{section-application}

In this section, we will give our third main result on distinguishing Hecke eigenforms by using the central values of a family of the twisted spinor $L$-functions. Recall that $\sigma_1$ is the standard representation of the dual group $\GL_1(\C)=\C^\times$ and $\pi_F$ is the associated cuspidal automorphic representation of $\GSp_4(\A)$ with trivial central character corresponding to a Hecke eigenform $F\in  \mathcal{S}_{k, j}^{\mathrm{new}}(K(N))$. Let $L(s,\pi_{F}\times\chi,\rho_4\otimes\sigma_1)$ be the twisted spinor $L$-function defined as in \eqref{finite part of twisted spinor L function}. Recall that $\Gamma_2=\Sp_4(\Z)$. Then we will show the following theorem.

\begin{theorem}\label{main thm 3}
Let $F_i\in\mathcal{S}_{k_i}(\Gamma_2), i=1, 2$, be two Hecke eigenforms. Suppose that for almost all primitive characters $\chi$ of squarefree conductor, we have
\begin{equation}\label{main assumption}
L(1/2,\pi_{F_1}\times\chi,\rho_4\otimes\sigma_1)=L(1/2,\pi_{F_2}\times\chi,\rho_4\otimes\sigma_1),
\end{equation}
then $k_1=k_2$ and $F_1=c\cdot F_2$ for some nonzero constant $c$.      
\end{theorem}

\subsection{The estimate of an average of twisted \texorpdfstring{$L$}{L}-values}

In order to prove Theorem~\ref{main thm 3}, we will need the technical result below (i.e., Theorem~\ref{Sketch of RY}), which is a generalized version of \cite[Theorem~2]{RadziwillYang2023}. To state this result properly, we introduce some notation and definitions first.
\begin{definition}\label{Siegel-Walfisz}
We say that a sequence $\{\alpha(n)\}_{n\geq 1}$ are \textit{Siegel-Walfisz} of level $\kappa>0$ if for every $x\geq 10$, and $(a, q)=1$ with $q\leq (\log x)^\kappa, |t|\leq (\log x)^\kappa$ we have
\begin{equation}
\sum_{\substack{p\leq x\\
p\equiv a\,(\mo q)}}\alpha(p)p^{it}\ll_A\frac{x}{(\log x)^A}
\end{equation}
for any given $A>10$. 
\end{definition}
Note that by the work of \cite{Brumley2006} the sequence $\{a_{F}(n)\}_{n\geq 1}$ are \textit{Siegel-Walfisz} of level $1>\kappa>0$, where $a_{F}(n)$ are defined as in \eqref{finite part of spinor L function}, i.e., the Dirichlet coefficients of the spinor $L$-function.
\begin{definition}\label{defn of Q delta nu}
Let $\kappa, \delta,\nu$ be positive numbers and $Q$ a large number. Set 
\begin{equation}\label{defn of P1 and P2}
P_1=(\log Q)^{\kappa\nu}\quad\text{and}\quad P_2=(\log Q)^{10000}.
\end{equation}
Then we define the set $\mathcal{Q}_{\delta,\nu}\subset [Q/16, 16Q]$ of squarefree integers of the form $p_1p_2m$ such that
\begin{enumerate}[(1)]
    \item 
    $p_1$ is a prime satisfying $p_1\sim P_1;$
    \item 
    $p_2$ is a prime satisfying $p_2\sim P_2;$
    \item 
    $m$ is squarefree satisfying $m\sim Q/(P_1P_2)$. In addition, $m$ has at most $(\delta\log\log Q+ 10)$ distinct prime factors and all of them are larger than $(\log Q)^{20000}$.
\end{enumerate}
Here, the notion $a\sim A$ means that $a\in[A,2A)$. 
\end{definition}
Now we state the required technical result as follows.
\begin{theorem}\label{Sketch of RY}
   Let $\pi_F$ be the cuspidal automorphic representation of $\GSp_4(\A)$ with trivial central character, which is corresponding to a Hecke eigenform $F\in\mathcal{S}_{k}^{\text{{\bf (G)}}}(\Gamma_2)$. For $\re(s)>3/2$, let
$L(s,\pi_{F},\rho_4)=\sum_{n\geq 1}a_{F}(n)n^{-s}$
be the corresponding spinor $L$-function. Let $\mathcal{Q}_{\delta,\nu}$ be the set of integers defined as in Definition~\ref{defn of Q delta nu} for some positive numbers $\delta, \nu$ and a large number $Q$. Let $p$ be a fixed prime. Then  there is a choice of $0<\delta<\nu<1$ with $\kappa\nu<1000$ such that
\begin{equation}
\begin{split}
&\sum_{q\in\mathcal{Q}_{\delta,\nu}}\,\sideset{}{^\star}\sum_{\chi\,(\mo q)}L(1/2,\pi_{F}\times\chi,\rho_4\otimes\sigma_1)\overline{\chi(p)}\\
&\qquad=\frac{a_{F}(p)}{p^{1/2}}Q^2(\log Q)^{-1+\delta-\delta\log\delta+o(1)}+O\big(Q^2(\log Q)^{-1+\delta-\delta\log\delta-2000\delta}\big)
\end{split}
\end{equation}
as $Q\to\infty$. Here, $\sideset{}{^\star}\sum$ means that the summation is over all primitive characters $\chi$. 
\end{theorem}

The proof of Theorem~\ref{Sketch of RY} follows the arguments used in the proof of \cite[Theorem~2]{RadziwillYang2023}, where the first moment of twisted $\GL_4(\A_{\Q})$ $L$-functions is studied. This is essentially due to the fact that for a Hecke eigenform $F\in\mathcal{S}_{k}^{\text{{\bf (G)}}}(\Gamma_2)$, one can associate to $F$ a cuspidal automorphic representation $\pi_F$ of $\GSp_4(\A)$ with trivial central character, and this representation $\pi_F$ can be transferred to a self-dual, symplectic, unitary, cuspidal automorphic representation $\Pi_{4,F}$ of $\GL_4(\A)$.
The new feature in Theorem~\ref{Sketch of RY} is the appearance of the factor $\chi(p)$ in the first moment of the twisted $L$-function $L(1/2,\pi_{F}\times\chi,\rho_4\otimes\sigma_1)$, where $p$ is a fixed prime. Since $p$ is fixed, the main term almost remains unchanged except an extra factor $a_{F}(p)/p^{1/2}$ appears (see Proposition~\ref{prop3}), and the error term is of the same order of magnitude. We therefore only sketch the proof of Theorem~\ref{Sketch of RY}. In particular, we first outline the proofs of the following three key propositions, which are analogous to \cite[Propositions~1-3]{RadziwillYang2023}.
\begin{proposition}\label{prop1}
 Let $\pi_F$ be the cuspidal automorphic representation of $\GSp_4(\A)$ with trivial central character, which is corresponding to a Hecke eigenform $F\in\mathcal{S}_{k}^{\text{{\bf (G)}}}(\Gamma_2)$.
  Let $\delta \geq 0$ be given.
  Let $\mathcal{R} \subset [R/4, 4R]$ and $\mathcal{S} \subset [S/4, 4S]$ be two sets of integers with $R \leq S$. Assume that for every $r \in \mathcal{R}$ and $s \in \mathcal{S}$, we have $(r,s)=1$, both $r$ and $s$ are squarefree, and moreover $\omega(r) \leq 10$ and $\omega(s) \leq \delta \log\log S + 10$.
  Let $\mathcal{Q}$ be the set of all integers that can be written as $r s$ with $r \in \mathcal{R}$ and
  $s \in \mathcal{S}$. Let $V$ be a smooth function compactly supported in $[1/100, 100]$.
  Finally set $Q \colonequals RS$ and assume that $S > R^{10}$. Let $p$ be a fixed prime. 
  Then, for any $M \geq 1$, 
  \begin{align*}
&\sum_{q\in\mathcal{Q}}\,\sideset{}{^\star}\sum_{\chi\,(\mo q)}\varepsilon(\pi_F,\chi)\sum_{m}\frac{a_F(m)\overline{\chi(mp)}}{m^{1/2}}V\left(\frac{m}{M}\right)
\\ &  \ll \| V \|_{\infty, 2} \cdot \Big (\sqrt{R M} Q (\log Q)^{-\frac{1}{2} + \frac{\delta}{2} + \frac{\delta}{2} \log \frac{1}{\delta}} +  Q^2 (\log Q)^{-1 + \delta + \delta \log \frac{1}{\delta}} \cdot R^{-1/4} (\log Q)^{C \delta} \Big ) 
  \end{align*}
  with $C > 10$ an absolute constant. Here, $\varepsilon(\pi_F,\chi)$ is the root number of $L(s,\pi_{F}\times\chi,\rho_4\otimes\sigma_1)$, and
  $$
  \| V \|_{\infty, 2} \colonequals \| V \|_{\infty} + \| V' \|_{\infty} + \| V'' \|_{\infty}.
  $$
\end{proposition}
\begin{proof}
 The desired estimate follows by a similar argument in \cite[\S~2, 3]{RadziwillYang2023}. In fact, the only difference is that we will replace the condition ``$v_1=v_2=\pm 1$" by ``$v_1=v_2=p$" when applying \cite[Lemma~7]{RadziwillYang2023} in the calculation. Notice that $p$ is a fixed prime, and the estimate is still valid.    
\end{proof}
\begin{proposition}\label{prop2}
 Let $\pi_F$ be the cuspidal automorphic representation of $\GSp_4(\A)$ with trivial central character, which is corresponding to a Hecke eigenform $F\in\mathcal{S}_{k}^{\text{{\bf (G)}}}(\Gamma_2)$.
  Let $0 < \kappa < 1$ be such that $a_F$ is Siegel-Walfisz of level $\kappa$.
  Let $0 < \nu, \delta < 1/1000$ be given.
  Let $\mathcal{Q}_{\delta,\nu}$ be the set of squarefree integers defined as in Definition~\ref{defn of Q delta nu}. Let $p$ be a fixed prime.
  Let $V$ be smooth and compactly supported in $[1/100, 100]$. Then, for any $N$ in the range,
  $$
  Q^2 (\log Q)^{-10^9} \leq N \leq Q^2 (\log Q)^{10^9}
  $$
  we have, for any $A > 10$, 
\begin{align*}
 &\sum_{q\in\mathcal{Q}_{\delta,\nu}}\,\sideset{}{^\star}\sum_{\chi\,(\mo q)}\sum_{n}\frac{a_F(n)\chi(n)\overline{\chi(p)}}{n^{1/2}}V\left(\frac{n}{N}\right)\\ 
 & \ll \| V \|_{\infty, 2} \cdot Q \sqrt{N} \left( \frac{e^{C / \delta}}{(\log Q)^{3/2}} \cdot (\log Q)^{4 \nu + 10 \delta + \delta \log \frac{1}{\delta}} + C(A, \nu, \kappa) (\log Q)^{-A} \right) 
\end{align*}
  with $C(A, \nu, \kappa)$ a constant depending only on $A$, $\nu$ and $\kappa$.  
\end{proposition}
\begin{proof}
   Following the proof of \cite[Proposition~2]{RadziwillYang2023}, we separate the inner $n$-sum into two parts: (i) each prime factor of $n$ is not in the range $[\exp(\log ^{\nu}N),N^{1/1000}]$; and (ii) the remaining case. And we point out that the proof of part (i) is an analogue to that of \cite[Proposition~4]{RadziwillYang2023}, and the proof of part (ii) is an analogue to that of \cite[Propositions~5-6]{RadziwillYang2023} once we replace all congruence conditions ``$\cdots\equiv \pm 1\,(\mo d)$'' by ``$\cdots\equiv p\,(\mo d)$ with $(p, d)=1$". Again, notice that $p$ is a fixed prime, and the estimate is still valid.
\end{proof}
\begin{proposition} \label{prop3}
 Let $\pi_F$ be the cuspidal automorphic representation of $\GSp_4(\A)$ with trivial central character, which is corresponding to a Hecke eigenform $F\in\mathcal{S}_{k}^{\text{{\bf (G)}}}(\Gamma_2)$.
  Let $V$ be a smooth function compactly supported in $[N/ 100, 100N]$.
  Let $\mathcal{Q}_{\delta,\nu}$ be the set of squarefree integers defined as in Definition~\ref{defn of Q delta nu}. We assume that $4 P_1 < 2 P_2 < (\log Q)^{20000}$. Let $p$ be a fixed prime. Then
  \begin{align*}
 &\sum_{q\in\mathcal{Q}_{\delta,\nu}}\,\sideset{}{^\star}\sum_{\chi\,(\mo q)}\sum_{n}\frac{a_F(n)\chi(n)\overline{\chi(p)}}{n^{1/2}}V\left(\frac{n}{N}\right)\\ 
 &= V \Big ( \frac{p}{N} \Big ) \frac{a_{F}(p)}{p^{1/2}}Q^2 (\log Q)^{\delta - 1 + \delta \log \frac{1}{\delta} + o(1)} + O\big( Q \sqrt{N} (\log Q)^2 \| V \|_{\infty} \big).
\end{align*}
\end{proposition}
\begin{proof}
This is identical to the proof of \cite[Proposition~3]{RadziwillYang2023}: by the orthogonality relation of primitive characters, we obtain
   \[\sum_{q\in\mathcal{Q}_{\delta,\nu}}\,\sideset{}{^\star}\sum_{\chi\,(\mo q)}\sum_{n}\frac{a_F(n)\chi(n)\overline{\chi(p)}}{n^{1/2}}V\left(\frac{n}{N}\right)=\sum_{cd\in\mathcal{Q}_{\delta,\nu}}\mu(c)\varphi(d)\sum_{n\equiv p(\mo{d})}\frac{a_F(n)}{n^{1/2}}V\left(\frac{n}{N}\right).\]
   We rewrite the congruence condition $n\equiv p(\mo d)$ as $n=p+\ell d$ with $\ell\in\mathbb{Z}_{\geq0}.$ By \cite[Lemmas~9-10]{RadziwillYang2023}, the $\ell=0$ term contributes the main term, while the remaining terms contribute to the error term.
   Notice that the $o(1)$-term in the exponent is determined by the choice of $\mathcal{Q}_{\delta,\nu}$.
\end{proof}
We are now ready to prove Theorem~\ref{Sketch of RY}.
\begin{proof}[\text{Proof of Theorem~\ref{Sketch of RY}}]
Recall that the sequence $\{a_{F_i}(n)\}_{n\geq 1}$ are \textit{Siegel-Walfisz} of level $1>\kappa>0$; see the remark below Definition~\ref{Siegel-Walfisz}. Let 
\begin{equation}\label{NplusMplusunbalanced}
N_{\nu}=Q^2(\log Q)^{1-200\nu}\quad\text{and}\quad M_{\nu}=Q^{2}(\log Q)^{-1+200\nu}.
\end{equation}
Moreover, for some $\epsilon>0$ we define
\begin{equation}\label{NplusMplus}
N_{\nu,\epsilon}^+\colonequals N_{\nu}(\log Q)^{\epsilon}\quad\text{and}\quad M_{\nu,\epsilon}^+\colonequals M_{\nu}(\log Q)^{\epsilon}.
\end{equation}
Note that the analytic conductor of $L(1/2,\pi_F\times\chi,\rho_4\otimes\sigma_1)$ is $\asymp Q^4$ and $N_{\nu}M_{\nu}=Q^4$; see \eqref{NplusMplusunbalanced}. Then we have the (slightly unbalanced) approximate functional equation
\begin{equation}
L(1/2,\pi_F\times\chi,\rho_4\otimes\sigma_1)=\sum_n\frac{a_F(n)\chi(n)}{n^{1/2}}W_1\left(\frac{n}{N_{\nu}}\right)+\varepsilon(\pi_F,\chi)\sum_{m}\frac{a_{F}(m)\overline{\chi(m)}}{m^{1/2}}W_2\left(\frac{m}{M_{\nu}}\right).
\end{equation}
Here, for any given $A>10$, $W_i, i=1,2$, are smooth functions satisfying
\begin{equation}
W_i(x)=\begin{cases}
1+O_A(x^A)&\mbox{for $x<1$},\\
O_A(x^{-A})&\mbox{for $x>1$},
\end{cases}
\end{equation}
and the root number  $\varepsilon(\pi_F,\chi)$ of $L(s,\pi_{F}\times\chi,\rho_4\otimes\sigma_1)$ is defined by
\begin{equation}
\varepsilon(\pi_F,\chi)=c_{\pi_F}\frac{\tau(\chi)^4}{q^2},
\end{equation}
where $|c_{\pi_F}|=1$ and $\tau(\chi)$ is the Gauss sum associated to $\chi$; see \cite{LuoWenzhi2005}.\footnote{If we consider paramodular newforms of level $N$, the term $\chi(N)$ will show up in $\varepsilon(\pi_F,\chi)$. However, the proof will be very similar in this case.}

For our purpose, we also introduce a partition of unity on $n$ and $m$ sums. More precisely, let $V$ be a smooth function compactly supported in $[1/4,4]$ such that for all positive number $x$,
\begin{equation}
1=\sum\limits_NV\left(\frac{x}{N}\right).
\end{equation}
Here, $N$ runs over powers of two. Then by the large sieve and Rankin-Selberg we get
\begin{equation}
    \sum_{q\in\mathcal{Q}_{\delta,\nu}}\,\sideset{}{^\star}\sum_{\chi\,(\mo q)}\sum_n\frac{a_F(n)\chi(n)\overline{\chi(p)}}{n^{1/2}}W_1\left(\frac{n}{N_{\nu}}\right)V\left(\frac{n}{N}\right)\ll_{A, \epsilon}Q^2(\log Q)^{-A}
\end{equation}
for any $A>10, V$ smooth and compactly supported away from $0$ and $N>N_{\nu, \epsilon}^+$. Similarly,
\begin{equation}
    \sum_{q\in\mathcal{Q}_{\delta,\nu}}\,\sideset{}{^\star}\sum_{\chi\,(\mo q)}\varepsilon(\pi_F,\chi)\sum_{m}\frac{a_F(m)\overline{\chi(mp)}}{m^{1/2}}W_2\left(\frac{m}{M_{\nu}}\right)V\left(\frac{m}{M}\right)\ll_{A, \epsilon}Q^2(\log Q)^{-A}
\end{equation}
for any $A>10, V$ smooth and compactly supported away from $0$ and $M>M_{\nu, \epsilon}^+$. It follows that
\begin{equation}\label{EqnfinitepartBound}
\sum_{q\in\mathcal{Q}_{\delta,\nu}}\,\sideset{}{^\star}\sum_{\chi\,(\mo q)}L(1/2,\pi_F\times\chi,\rho_4\otimes\sigma_1)\overline{\chi(p)}=\sum_{N\leq N_{\nu,\epsilon}^+}\mathcal{F}_N+\sum_{M\leq M_{\nu,\epsilon}^+}\mathcal{D}_M+O_{A,\epsilon}\left(Q^2(\log Q)^{-A}\right)
\end{equation}
with $N,M$ running over the powers of two, and where
\begin{equation}\label{FN term}
\mathcal{F}_N\colonequals\sum_{q\in\mathcal{Q}_{\delta,\nu}}\,\sideset{}{^\star}\sum_{\chi\,(\mo q)}\sum_{n}\frac{a_F(n)\chi(n)\overline{\chi(p)}}{n^{1/2}}W_1\left(\frac{n}{N_{\nu}}\right)V\left(\frac{n}{N}\right),
\end{equation}
and
\begin{equation}\label{DM term}
\mathcal{D}_M\colonequals\sum_{q\in\mathcal{Q}_{\delta,\nu}}\,\sideset{}{^\star}\sum_{\chi\,(\mo q)}\varepsilon(\pi_F,\chi)\sum_{m}\frac{a_F(m)\overline{\chi(mp)}}{m^{1/2}}W_2\left(\frac{m}{M_{\nu}}\right)V\left(\frac{m}{M}\right)
.
\end{equation}
Next, we estimate \eqref{EqnfinitepartBound} in each of the ranges by applying Propositions~\ref{prop1}–\ref{prop3}. 
\begin{enumerate}
    \item For $\mathcal{F}_N$ in the range $N\leq Q^2(\log Q)^{-10^9}$, by Proposition~\ref{prop3} we have
\begin{equation}
\sum_{N\leq Q^2(\log Q)^{-10^9}}\mathcal{F}_N=\frac{a_F(p)}{p^{1/2}}Q^2(\log Q)^{-1+\delta-\delta\log\delta+o(1)}+O(Q^2(\log Q)^{-100})
\end{equation}
as $Q\to\infty$.
\item  For $\mathcal{F}_N$ in the range $Q^2(\log Q)^{-10^9}\leq N\leq N_{\nu,\epsilon}^{+}$, this contributes only to the error term. Indeed, we first recall that the sequence $\{a_{F}(n)\}_{n\geq 1}$ are \textit{Siegel-Walfisz} of some level $1>\kappa>0$. There are $\log\log Q$ dy-adic ranges
to consider in this case. According to Proposition~\ref{prop2} their total contribution is
\begin{align*}
& \ll e^{O(1 / \delta)} \cdot \frac{\sqrt{N_{\nu, \varepsilon}^{+}}\,Q}{(\log Q)^{3/2}} (\log Q)^{4 \nu + \delta \log \frac{1}{\delta} + 10 \delta + 2 \varepsilon}  + C(\nu, \kappa) Q^2 (\log Q)^{-10} \\
& \ll e^{O(1 / \delta)} \cdot Q^2(\log Q)^{-1 - 100 \nu + \varepsilon + \delta \log \frac{1}{\delta} + 10 \delta + 2 \varepsilon} + C(\nu, \kappa) Q^2 (\log Q)^{-10}
\end{align*}
with $C(\nu, \kappa)$ a constant depending only on $\nu$ and $\kappa$. Notice that $p$ is a fixed prime, the above is negligible compared to the main term
$\frac{a_F(p)}{p^{1/2}}Q^2 (\log Q)^{-1 + \delta + \delta \log \frac{1}{\delta} + o(1)}$ for any fixed $0< \varepsilon<\delta<\nu<1/1000$.
\item  For $\mathcal{D}_M$ in the ranges $M\leq Q^2(\log Q)^{-10^9}$ and $Q^2(\log Q)^{-10^9}\leq M\leq M_{\nu,\epsilon}^{+}$, both also contribute only to the error term. This follows from Proposition~\ref{prop1} and the same arguments in \cite[\S~1.3 and \S~1.4]{RadziwillYang2023}, respectively. All the estimates are the same due to that $p$ is a fixed prime.
\end{enumerate}
This completes the proof.
\end{proof}

\subsection{Proof of Theorem~\ref{main thm 3}}

With Theorem~\ref{Sketch of RY} in hand, we are now ready to give the proof of Theorem~\ref{main thm 3}.

\begin{proof}[Proof of Theorem~\ref{main thm 3}]
As in the proofs of Theorems~\ref{main thm 1} and \ref{main thm 2}, we will consider three cases: i) both of them are of type {\bf (G)}; ii) one is of type {\bf (G)} while the other is of type {\bf (P)}; and iii) both of them are of type {\bf (P)}. Let $q$ be the conductor of primitive character $\chi$ as in the assumption.

i) This part is based on the pre-print \cite{RadziwillYang2023}. For $i=1, 2$, let $\pi_{F_i}$ be the cuspidal automorphic representation of $\GSp_4(\A)$ of type {\bf (G)} corresponding to $F_i$. Then it corresponds to a cuspidal automorphic representation $\Pi_{4,F_i}$ of $\GL_4(\A)$ such that  (see \eqref{finite part of spinor L function})
\[L(s,\pi_{F_i},\rho_4)=L(s,\Pi_{4,F_i})=\sum_{n=1}^{\infty}\frac{a_{F_i}(n)}{n^s},\quad\text{for } \re(s)>3/2.\] 
Again, recall that the sequence $\{a_{F_i}(n)\}_{n\geq 1}$ are \textit{Siegel-Walfisz} of level $\kappa>0$. Let $\mathcal{Q}_{\delta,\nu}$ be the set of integers defined as in Definition~\ref{defn of Q delta nu} for some positive numbers $\delta, \nu$ and a large number $Q$.
Let $p$ be a fixed prime. Then by Theorem~\ref{Sketch of RY}, there is a choice of $0<\delta<\nu<1$ with $\kappa\nu<1000$ such that
\begin{equation}\label{EqnAsymAveLchip}
\begin{split}
&\sum_{q\in\mathcal{Q}_{\delta,\nu}}\,\sideset{}{^\star}\sum_{\chi\,(\mo q)}L(1/2,\pi_{F_i}\times\chi,\rho_4\otimes\sigma_1)\overline{\chi(p)}\\
&\qquad=\frac{a_{F_i}(p)}{p^{1/2}}Q^2(\log Q)^{-1+\delta-\delta\log\delta+o(1)}+O(Q^2(\log Q)^{-1+\delta-\delta\log\delta-2000\delta})
\end{split}
\end{equation}
as $Q\to \infty$. Indeed, by \cite[Lemma~10]{RadziwillYang2023}, the $o(1)$ term in the exponent of $\log(Q)$ in \eqref{EqnAsymAveLchip} is determined by the set $\mathcal{Q}_{\delta,\nu}$; see also the proof of Proposition~\ref{prop3}. Note that $\pi_{F_1}$ and $\pi_{F_2}$ are cuspidal automorphic representations of $\GSp_4(\A)$ of type {\bf (G)}, and as such can be transferred to cuspidal automorphic representations of $\GL_4(\A_{\Q})$. Consequently, the same set $\mathcal{Q}_{\delta,\nu}$ can be chosen for both $\pi_{F_1}$ and $\pi_{F_2}$.
In this case, by the assumption \eqref{main assumption} we can show that $a_{F_1}(p)=a_{F_2}(p)$. Since $p$ is chosen arbitrarily, by Theorem~\ref{main thm 1} we have $k_1=k_2$ and $F_1=c\cdot F_2$ for some nonzero constant $c$ as desired.

ii) We will again show that this case can never happen. Suppose not. Let $F_1$ be of type {\bf (P)} and $F_2$ be of type {\bf (G)}. Then there exists $f_1\in\mathcal{S}_{2k_1-2}(\SL_2(\Z))$ such that $F_1=F_{f_1}$ and
\begin{equation}
L(s,\pi_{F_1},\rho_4)=\zeta(s-1/2)\zeta(s+1/2)L(s, \pi_{f_1}).
\end{equation}
Let $\chi$ be an arbitrary even primitive character. Then we can see that
\begin{equation}
L(1/2,\pi_{F_1}\times\chi,\rho_4\otimes\sigma_1)=L(0,\chi)L(1,\chi)L(1/2,\pi_{f_1}\times\chi)=0
\end{equation}
by the functional equation of $L(s,\chi)$. However, for $F_2$ of type {\bf (G)}, it will correspond to a self-dual, unitary, cuspidal automorphic representation of $\GL_4(\A)$  with trivial central character, say $\Pi_{4, F_2}$. Then taking $\pi=\Pi_{4, F_2}$ in the pre-print \cite[Theorem~1]{RadziwillYang2023}, there are infinitely many even primitive character $\chi$ such that $L(1/2,\Pi_{4, F_2}\times\chi)\neq 0$, which leads to the desired contradiction. 

iii) In this case, we actually just require that the conductor $q$ of primitive character $\chi$ is a prime number in the assumption. Suppose that $F_i=F_{f_i}, i=1, 2$, with $f_i\in\mathcal{S}_{2k_i-2}(\SL_2(\Z))$. Moreover, let $\pi_{f_i}$ be the cuspidal automorphic representation of $\GL_2(\A)$ associated to $f_i$. Then for any odd  primitive character $\chi$ of conductor $q$, we have $L(0,\chi)L(1,\chi)\neq 0$. Then the assumption \eqref{main assumption} is reduced to
\begin{equation}\label{main assumption in both P types}
L(1/2,\pi_{f_1}\times\chi)=L(1/2,\pi_{f_2}\times\chi),
\end{equation}
due to the relation (see \eqref{SK lift L function in terms of f and zeta} for the non-twisted case)
\begin{equation}
L(1/2,\pi_{F_i}\times\chi,\rho_4\otimes\sigma_1)=L(0,\chi)L(1,\chi)L(1/2,\pi_{f_i}\times\chi).
\end{equation}
Let $p$ be a fixed prime. For our purpose, we also assume that $q$ is larger than $p$. Consider the sum 
\begin{equation}\label{main sum 2}
\mathcal{S}_{f_i}(p)=\sideset{}{^\star}\sum_{\substack{\chi\,(\mo q)\\ \chi(-1)=-1}}L(1/2,\pi_{f_i}\times\chi)\overline{\chi(p)}
\end{equation}
following the work in \cite{LuRa1997}. For $X\geq 1$, we have the following approximate functional equation
\begin{equation}\label{aFE SK case}
L(1/2,\pi_{f_i}\times\chi)=\sum_{n\geq1}\frac{\lambda_{f_i}(n)\chi(n)}{n^{1/2}}U\left(\frac{n}{qX}\right)-\frac{\tau(\chi)^2}{q}\sum_{n\geq1}\frac{\lambda_{f_i}(n)\overline{\chi(n)}}{n^{1/2}}U\left(\frac{nX}{q}\right).
\end{equation}
Here, $\lambda_{f_i}(n)$ are the Hecke eigenvalues associated to $f_i$ and
\begin{equation}
U(y)=\frac{1}{2\pi i}\int_{(2)}\frac{\Gamma(u+k_i-1)}{\Gamma(k_i-1)}G(u)(2\pi y)^{-u}\frac{\,du}{u},
\end{equation}
where $G(u)$ is some even function with exponential decay. By \cite[Proposition~5.4]{IwKo2004}, we can show that
\begin{equation}\label{eq.behavior of the truncation function}
\begin{split}
   U(y)&=1+O(y^{1/2}) \quad \text{as } y\to0, \\
y^jU^{(j)}(y)&\ll_A y^{-A}\quad \text{as } y\to\infty. 
\end{split}
\end{equation}

Let $\bar{p}$ be the multiplicative inverse of $p$ modulo $q$. Inserting the approximate functional equation \eqref{aFE SK case} into \eqref{main sum 2}, we obtain
\begin{equation}
\mathcal{S}_{f_i}(p)=I-II,
\end{equation}
where
\begin{align}
I=&\sum_{\substack{n\geq 1\\(n, q)=1}}\frac{\lambda_{f_i}(n)}{n^{1/2}}U\left(\frac{n}{qX}\right)\sideset{}{^\star}\sum_{\substack{\chi\,(\mo q)\\ \chi(-1)=-1}}\chi(n)\overline{\chi(p)},\label{I in SK lifts}\\
II=&\sum_{\substack{n\geq 1\\(n, q)=1}}\frac{\lambda_{f_i}(n)}{n^{1/2}}U\left(\frac{nX}{q}\right)\frac{1}{q}\sideset{}{^\star}\sum_{\substack{\chi\,(\mo q)\\ \chi(-1)=-1}}\overline{\chi(n)}\overline{\chi(p)}\tau(\chi)^2.\label{II in SK lifts}
\end{align}
Moreover, we will show that the second term $II$ actually contributes to the error term. First, by \cite[(3.8)]{IwKo2004} we have
\begin{equation}\label{IK200438}
 \sideset{}{^\star}\sum_{\chi\,(\mo q)}\chi(m)\chi(n)=\sum_{d\mid (mn-1,q)}\varphi(d)\mu\left(\frac{q}{d}\right),\quad (mn,q)=1.
\end{equation} 
On the other hand, we have
\begin{equation}\label{rewriteoddcharacter}
\sideset{}{^\star}\sum_{\substack{\chi\,(\mo q)\\ \chi(-1)=-1}}\chi(m)\chi(n)=\sideset{}{^\star}\sum_{\chi\,(\mo q)}\frac{\chi(1)-\chi(-1)}{2}\chi(m)\chi(n).
\end{equation}
Then by $\overline{\chi(p)}=\chi(\bar{p})$ and \eqref{IK200438}-\eqref{rewriteoddcharacter}, the double summation \eqref{I in SK lifts} becomes
\begin{equation}\label{eq.term I in SK}
I=\frac{1}{2}\sum_{d\mid q}\varphi(d)\mu\left(\frac{q}{d}\right)\left\lbrace\sum_{\substack{n\geq1\\ n\equiv p\,(\mo d)}}\frac{\lambda_{f_i}(n)}{n^{1/2}}U\left(\frac{n}{qX}\right)-\sum_{\substack{n\geq1\\ n\equiv -p\,(\mo d)}}\frac{\lambda_{f_i}(n)}{n^{1/2}}U\left(\frac{n}{qX}\right)\right\rbrace.
\end{equation}
Notice that $p$ is fixed. We choose $d$ large enough. Then by Deligne's bound $\lambda_{f_i}(n)\ll n^{\epsilon}$ and \eqref{eq.behavior of the truncation function}, we obtain
\begin{align*}
\sum_{\substack{n\geq1\\ n\equiv p\,(\mo d)}}\frac{\lambda_{f_i}(n)}{n^{1/2}}U\left(\frac{n}{qX}\right)&=\frac{\lambda_{f_i}(p)}{p^{1/2}}U\left(\frac{p}{qX}\right)+\sum_{\ell\geq1}\frac{\lambda_{f_i}(p+\ell d)}{(p+\ell d)^{1/2}}U\left(\frac{p+\ell d}{qX}\right)\\
&=\frac{\lambda_{f_i}(p)}{p^{1/2}}U\left(\frac{p}{qX}\right)+O\left(\frac{X^{\epsilon}}{d^{1/2}}\sum_{\ell\ll \frac{(qX)^{1+\epsilon}}{d}}\frac{1}{d^{1/2}}\right)\\
&=\frac{\lambda_{f_i}(p)}{p^{1/2}}U\left(\frac{p}{qX}\right)+O\left(\frac{(qX)^{\frac{1}{2}+\epsilon}}{d}\right),
\end{align*}
and
\[\sum_{\substack{n\geq1\\ n\equiv -p\,(\mo d)}}\frac{\lambda_{f_i}(n)}{n^{1/2}}U\left(\frac{n}{qX}\right)=\sum_{\ell\geq1}\frac{\lambda_{f_i}(-p+\ell d)}{(-p+\ell d)^{1/2}}U\left(\frac{-p+\ell d}{qX}\right)=O\left(\frac{(qX)^{\frac{1}{2}+\epsilon}}{d}\right).\]
Inserting them into \eqref{eq.term I in SK}, we find
\begin{equation}\label{I final}
I=\frac{1}{2}\frac{\lambda_{f_i}(p)}{p^{1/2}}U\left(\frac{p}{qX}\right)\sum_{d\mid q}\varphi(d)\mu\left(\frac{q}{d}\right)+O\left(\tau(q)(qX)^{1/2+\epsilon}\right),
\end{equation}
where $\tau(q)$ is the divisor function of $q$. As for $II$, we can show that the inner summation becomes
\begin{equation}\label{useful est for II}
\begin{split}
\sideset{}{^\star}\sum_{\substack{\chi\,(\mo q)\\ \chi(-1)=-1}}\overline{\chi(n)}\overline{\chi(p)}\tau(\chi)^2&=\frac{1}{2}\sum_{d\mid q}\varphi(d)\mu\left(\frac{q}{d}\right)\left\lbrace S\left(\frac{q^2}{d^2},np;d\right)-S\left(\frac{q^2}{d^2},-np;d\right)\right\rbrace\\
&\ll \tau(q)^2q^{3/2},
\end{split}
\end{equation}
which is similar to \cite[Lemma~2.15]{LuRa1997}. In fact, the bound is due to Weil's bound on the Kloosterman sum. Then plugging \eqref{useful est for II} into \eqref{II in SK lifts} we obtain
\begin{equation}\label{II final}
II=O(q^{1+\epsilon}/X^{1/2+\epsilon}).
\end{equation}
Notice that $p$ is fixed. We choose $X=q^{1/2}$ in \eqref{I final} and \eqref{II final}, and then apply \eqref{eq.behavior of the truncation function}, which implies 
\begin{equation}
\begin{split}
  \mathcal{S}_{f_i}(p)=I-II&=\frac{1}{2}\frac{\lambda_{f_i}(p)}{p^{1/2}}U\left(\frac{p}{qX}\right)\sum_{d\mid q}\varphi(d)\mu\left(\frac{q}{d}\right)+O(q^{3/4+\epsilon})\\
  &=\frac{1}{2}\frac{\lambda_{f_i}(p)}{p^{1/2}}\sum_{d\mid q}\varphi(d)\mu\left(\frac{q}{d}\right)+O(q^{3/4+\epsilon}).  
\end{split}
\end{equation}
Since $q$ is squarefree by definition, it follows from multiplicativity that
\begin{equation}\label{Eqntheinnersumnonzero}
    \sum_{d\mid q}\varphi(d)\mu\left(\frac{q}{d}\right)=\prod_{\substack{v\mid q\\v \text{ prime}}}\left(\sum_{d\mid v}\varphi(d)\mu\left(\frac{v}{d}\right)\right)=\prod_{\substack{v\mid q\\v \text{ prime}}}(v-2).
\end{equation}
In particular, the sum \eqref{Eqntheinnersumnonzero} is nonzero whenever $q$ is odd. Indeed, by the definition of $q\in \mathcal{Q}_{\delta,\nu}$ (see Definition~\ref{defn of Q delta nu}), all prime factors of $q$ are sufficiently large. 

Finally, by \eqref{main assumption in both P types} and let $q\to\infty$, we have $\lambda_{f_1}(p)=\lambda_{f_2}(p)$ for almost all primes $p$. This shows that $f_1$ is a constant multiple of $f_2$ by the strong multiplicity one theorem for $\GL_2$. Therefore, we have $a_{F_1}(p)=a_{F_2}(p)$ for almost all primes $p$ and hence again by Theorem~\ref{main thm 1} we obtain $F_1=c\cdot F_2$ for some nonzero constant $c$ as expected.
\end{proof}

\bibliographystyle{alpha}
\bibliography{SMO.bib}

\begin{thebibliography}{BCGP21}

\bibitem[And74]{Andrianov1974}
A.~N. Andrianov.
\newblock Euler products that correspond to {S}iegel's modular forms of genus
  {$2$}.
\newblock {\em Uspehi Mat. Nauk}, 29(3 (177)):43--110, 1974.

\bibitem[Art04]{Arthur2004}
James Arthur.
\newblock Automorphic representations of {${\rm GSp(4)}$}.
\newblock In {\em Contributions to automorphic forms, geometry, and number
  theory}, pages 65--81. Johns Hopkins Univ. Press, Baltimore, MD, 2004.

\bibitem[AS10]{AvdispahicSmajlovic2010}
Muharem Avdispahi\'c and Lejla Smajlovi\'c.
\newblock On the {S}elberg orthogonality for automorphic {$L$}-functions.
\newblock {\em Arch. Math. (Basel)}, 94(2):147--154, 2010.

\bibitem[BCGP21]{BoxerCalegariGeePilloni2021}
George Boxer, Frank Calegari, Toby Gee, and Vincent Pilloni.
\newblock Abelian surfaces over totally real fields are potentially modular.
\newblock {\em Publ. Math. Inst. Hautes \'{E}tudes Sci.}, 134:153--501, 2021.

\bibitem[BK14]{BrumerKramer2014}
Armand Brumer and Kenneth Kramer.
\newblock Paramodular abelian varieties of odd conductor.
\newblock {\em Trans. Amer. Math. Soc.}, 366(5):2463--2516, 2014.

\bibitem[BK19]{BrumerKramer2019}
Armand Brumer and Kenneth Kramer.
\newblock Corrigendum to ``{P}aramodular abelian varieties of odd conductor".
\newblock {\em Trans. Amer. Math. Soc.}, 372(3):2251--2254, 2019.

\bibitem[Bru06]{Brumley2006}
Farrell Brumley.
\newblock Effective multiplicity one on {${\rm GL}_N$} and narrow zero-free
  regions for {R}ankin-{S}elberg {$L$}-functions.
\newblock {\em Amer. J. Math.}, 128(6):1455--1474, 2006.

\bibitem[Bum97]{Bump1997}
Daniel Bump.
\newblock {\em Automorphic forms and representations}, volume~55 of {\em
  Cambridge Studies in Advanced Mathematics}.
\newblock Cambridge University Press, Cambridge, 1997.

\bibitem[Cas73]{Casselman1973}
William Casselman.
\newblock On some results of {A}tkin and {L}ehner.
\newblock {\em Math. Ann.}, 201:301--314, 1973.

\bibitem[CG13]{CG-irred}
Frank Calegari and Toby Gee.
\newblock Irreducibility of automorphic {G}alois representations of {$GL(n)$},
  {$n$} at most 5.
\newblock {\em Ann. Inst. Fourier (Grenoble)}, 63(5):1881--1912, 2013.

\bibitem[DK00]{Ram2000}
W.~Duke and E.~Kowalski.
\newblock A problem of {L}innik for elliptic curves and mean-value estimates
  for automorphic representations.
\newblock {\em Invent. Math.}, 139(1):1--39, 2000.
\newblock With an appendix by Dinakar Ramakrishnan.

\bibitem[FPRS13]{FarmerPitaleRyanSchmidt2013arxiv}
David~W Farmer, Ameya Pitale, Nathan~C Ryan, and Ralf Schmidt.
\newblock Multiplicity one for {$L$}-functions and applications.
\newblock {\em arXiv preprint arXiv:1305.3972}, 2013.

\bibitem[GH93]{GoldfeldHoffstein1993}
Dorian Goldfeld and Jeffrey Hoffstein.
\newblock On the number of {F}ourier coefficients that determine a modular
  form.
\newblock In {\em A tribute to {E}mil {G}rosswald: number theory and related
  analysis}, volume 143 of {\em Contemp. Math.}, pages 385--393. Amer. Math.
  Soc., Providence, RI, 1993.

\bibitem[Ghi11]{Gh2011}
Alexandru Ghitza.
\newblock Distinguishing {H}ecke eigenforms.
\newblock {\em Int. J. Number Theory}, 7(5):1247--1253, 2011.

\bibitem[HSBT10]{HST2010}
Michael Harris, Nick Shepherd-Barron, and Richard Taylor.
\newblock A family of {C}alabi-{Y}au varieties and potential automorphy.
\newblock {\em Ann. of Math. (2)}, 171(2):779--813, 2010.

\bibitem[IK04]{IwKo2004}
Henryk Iwaniec and Emmanuel Kowalski.
\newblock {\em Analytic number theory}, volume~53 of {\em American Mathematical
  Society Colloquium Publications}.
\newblock American Mathematical Society, Providence, RI, 2004.

\bibitem[JPSS83]{JacquetPiatetski-ShapiroShalika1983}
H.~Jacquet, I.~I. Piatetskii-Shapiro, and J.~A. Shalika.
\newblock Rankin-{S}elberg convolutions.
\newblock {\em Amer. J. Math.}, 105(2):367--464, 1983.

\bibitem[JS81a]{JacquetShalika198102}
H.~Jacquet and J.~A. Shalika.
\newblock On {E}uler products and the classification of automorphic forms.
  {II}.
\newblock {\em Amer. J. Math.}, 103(4):777--815, 1981.

\bibitem[JS81b]{JacquetShalika198101}
H.~Jacquet and J.~A. Shalika.
\newblock On {E}uler products and the classification of automorphic
  representations. {I}.
\newblock {\em Amer. J. Math.}, 103(3):499--558, 1981.

\bibitem[Kim03]{Kim2003}
Henry~H. Kim.
\newblock Functoriality for the exterior square of {${\rm GL}_4$} and the
  symmetric fourth of {${\rm GL}_2$}.
\newblock {\em J. Amer. Math. Soc.}, 16(1):139--183, 2003.
\newblock With appendix 1 by Dinakar Ramakrishnan and appendix 2 by Kim and
  Peter Sarnak.

\bibitem[KMS21]{KumarMeherShankhadhar2021}
Arvind Kumar, Jaban Meher, and Karam~Deo Shankhadhar.
\newblock Strong multiplicity one for {S}iegel cusp forms of degree two.
\newblock {\em Forum Math.}, 33(5):1157--1167, 2021.

\bibitem[Lau05]{Lau2005}
G\'{e}rard Laumon.
\newblock Fonctions z\^{e}tas des vari\'{e}t\'{e}s de {S}iegel de dimension
  trois.
\newblock Number 302, pages 1--66. 2005.
\newblock Formes automorphes. II. Le cas du groupe ${\rm{GSp}}(4)$.

\bibitem[LR97]{LuRa1997}
Wenzhi Luo and Dinakar Ramakrishnan.
\newblock Determination of modular forms by twists of critical {$L$}-values.
\newblock {\em Invent. Math.}, 130(2):371--398, 1997.

\bibitem[LRS99]{LuoRudnickSarnak1999}
Wenzhi Luo, Ze\'{e}v Rudnick, and Peter Sarnak.
\newblock On the generalized {R}amanujan conjecture for {${\rm GL}(n)$}.
\newblock In {\em Automorphic forms, automorphic representations, and
  arithmetic ({F}ort {W}orth, {TX}, 1996)}, volume~66 of {\em Proc. Sympos.
  Pure Math.}, pages 301--310. Amer. Math. Soc., Providence, RI, 1999.

\bibitem[Luo05]{LuoWenzhi2005}
Wenzhi Luo.
\newblock Nonvanishing of {$L$}-functions for {${\rm GL}(n,{\bf A}_{\bf Q})$}.
\newblock {\em Duke Math. J.}, 128(2):199--207, 2005.

\bibitem[LW09]{LiuWang2009}
Jianya Liu and Yonghui Wang.
\newblock A theorem on analytic strong multiplicity one.
\newblock {\em J. Number Theory}, 129(8):1874--1882, 2009.

\bibitem[LWY05]{LiWaYe2005}
Jianya Liu, Yonghui Wang, and Yangbo Ye.
\newblock A proof of {S}elberg's orthogonality for automorphic {$L$}-functions.
\newblock {\em Manuscripta Math.}, 118(2):135--149, 2005.

\bibitem[Mor85]{Moreno1985}
Carlos~J. Moreno.
\newblock Analytic proof of the strong multiplicity one theorem.
\newblock {\em Amer. J. Math.}, 107(1):163--206, 1985.

\bibitem[Mur97]{RamMurty1997}
M.~Ram Murty.
\newblock Congruences between modular forms.
\newblock In {\em Analytic number theory ({K}yoto, 1996)}, volume 247 of {\em
  London Math. Soc. Lecture Note Ser.}, pages 309--320. Cambridge Univ. Press,
  Cambridge, 1997.

\bibitem[PS79]{Piatetski-Shapiro1979}
I.~I. Piatetski-Shapiro.
\newblock Multiplicity one theorems.
\newblock In {\em Automorphic forms, representations and {$L$}-functions
  ({P}roc. {S}ympos. {P}ure {M}ath., {O}regon {S}tate {U}niv., {C}orvallis,
  {O}re., 1977), {P}art 1}, volume XXXIII of {\em Proc. Sympos. Pure Math.},
  pages 209--212. Amer. Math. Soc., Providence, RI, 1979.

\bibitem[PSS14]{PitaleSahaSchmidt2014}
Ameya Pitale, Abhishek Saha, and Ralf Schmidt.
\newblock Transfer of {S}iegel cusp forms of degree 2.
\newblock {\em Mem. Amer. Math. Soc.}, 232(1090):vi+107, 2014.

\bibitem[PY15]{PoorYuen2015}
Cris Poor and David~S. Yuen.
\newblock Paramodular cusp forms.
\newblock {\em Math. Comp.}, 84(293):1401--1438, 2015.

\bibitem[RS96]{RudnickSarnak1996}
Ze\'{e}v Rudnick and Peter Sarnak.
\newblock Zeros of principal {$L$}-functions and random matrix theory.
\newblock {\em Duke Math. J.}, 81(2):269--322, 1996.
\newblock A celebration of John F. Nash, Jr.

\bibitem[RS06]{RobertsSchmidt2006}
Brooks Roberts and Ralf Schmidt.
\newblock On modular forms for the paramodular groups.
\newblock In {\em Automorphic forms and zeta functions}, pages 334--364. World
  Sci. Publ., Hackensack, NJ, 2006.

\bibitem[RS07]{RobertsSchmidt2007}
Brooks Roberts and Ralf Schmidt.
\newblock {\em Local newforms for {GS}p(4)}, volume 1918 of {\em Lecture Notes
  in Mathematics}.
\newblock Springer, Berlin, 2007.

\bibitem[RSY21]{RoySchmidtYi2021}
Manami Roy, Ralf Schmidt, and Shaoyun Yi.
\newblock On counting cuspidal automorphic representations for {$\rm GSp(4)$}.
\newblock {\em Forum Math.}, 33(3):821--843, 2021.

\bibitem[RW17]{RosnerWeissauer2017}
Mirko R\"{o}sner and Rainer Weissauer.
\newblock Multiplicity one for certain paramodular forms of genus two.
\newblock In {\em L-functions and automorphic forms}, volume~10 of {\em
  Contrib. Math. Comput. Sci.}, pages 251--264. Springer, Cham, 2017.

\bibitem[RY23]{RadziwillYang2023}
Maksym Radziwi{\l}{\l} and Liyang Yang.
\newblock {Non-vanishing of twists of $\text{GL}_4(\mathbb{A}_{\bf Q})$
  ${L}$-functions}.
\newblock {\em arXiv preprint arXiv:2304.09171}, 2023.

\bibitem[Sch17]{Schmidt2017}
Ralf Schmidt.
\newblock Archimedean aspects of {S}iegel modular forms of degree 2.
\newblock {\em Rocky Mountain J. Math.}, 47(7):2381--2422, 2017.

\bibitem[Sch18]{Schmidt2018}
Ralf Schmidt.
\newblock Packet structure and paramodular forms.
\newblock {\em Trans. Amer. Math. Soc.}, 370(5):3085--3112, 2018.

\bibitem[Sch20]{Schmidt2020}
Ralf Schmidt.
\newblock Paramodular forms in {CAP} representations of {${\rm GSp}(4)$}.
\newblock {\em Acta Arith.}, 194(4):319--340, 2020.

\bibitem[Sou87]{Soudry1987}
David Soudry.
\newblock A uniqueness theorem for representations of {${\rm GSO}(6)$} and the
  strong multiplicity one theorem for generic representations of {${\rm
  GSp}(4)$}.
\newblock {\em Israel J. Math.}, 58(3):257--287, 1987.

\bibitem[Stu87]{Sturm1987}
Jacob Sturm.
\newblock On the congruence of modular forms.
\newblock In {\em Number theory ({N}ew {Y}ork, 1984--1985)}, volume 1240 of
  {\em Lecture Notes in Math.}, pages 275--280. Springer, Berlin, 1987.

\bibitem[Tay93]{Tay1993}
Richard Taylor.
\newblock On the {$l$}-adic cohomology of {S}iegel threefolds.
\newblock {\em Invent. Math.}, 114(2):289--310, 1993.

\bibitem[VX18]{VilardiXue2018}
Trevor Vilardi and Hui Xue.
\newblock Distinguishing eigenforms of level one.
\newblock {\em Int. J. Number Theory}, 14(1):31--36, 2018.

\bibitem[Wei05]{Wei2005}
Rainer Weissauer.
\newblock Four dimensional {G}alois representations.
\newblock Number 302, pages 67--150. 2005.
\newblock Formes automorphes. II. Le cas du groupe ${\rm{GSp}}(4)$.

\bibitem[WY22]{wei2022distinguishing}
Zhining Wei and Shaoyun Yi.
\newblock On distinguishing siegel cusp forms of degree two.
\newblock {\em To appear in International Journal of Number Theory, arXiv
  preprint arXiv:2207.13234}, 2022.

\end{thebibliography}

\vspace{5ex}
\noindent Department of Mathematics, Ohio State University, Columbus, OH 43210, USA.

\noindent E-mail address: {\tt wangxiyuan928@gmail.com}

\vspace{2ex}
\noindent Department of Mathematics, Brown University, Providence, RI 02912, USA.

\noindent E-mail address: {\tt zhining\_wei@brown.edu}

\vspace{2ex}
\noindent Department of Mathematics, University of Arizona, Tucson, AZ 85721, USA.

\noindent E-mail address: {\tt panyan@arizona.edu}

\vspace{2ex}
\noindent School of Mathematical Sciences, Xiamen University, Xiamen, Fujian 361005, PR China.

\noindent E-mail address: {\tt yishaoyun926@xmu.edu.cn}

\end{document}